\documentclass[oneside,english]{amsart}
\usepackage[T1]{fontenc}
\usepackage[latin9]{inputenc}
\usepackage{color}
\usepackage{babel}
\usepackage{amstext}
\usepackage{amsthm}
\usepackage{amssymb}
\usepackage[unicode=true,pdfusetitle,
 bookmarks=true,bookmarksnumbered=false,bookmarksopen=false,
 breaklinks=false,pdfborder={0 0 1},backref=page,colorlinks=false]
 {hyperref}

\makeatletter
\numberwithin{equation}{section}
\numberwithin{figure}{section}
\theoremstyle{plain}
\newtheorem{thm}{\protect\theoremname}
  \theoremstyle{definition}
  \newtheorem{defn}[thm]{\protect\definitionname}
  \theoremstyle{plain}
  \newtheorem{prop}[thm]{\protect\propositionname}
  \theoremstyle{plain}
  \newtheorem{cor}[thm]{\protect\corollaryname}
 \newcommand\thmsname{\protect\theoremname}
 \newcommand\nm@thmtype{theorem}
 \theoremstyle{plain}
 
 \newenvironment{namedthm}[1][Undefined Theorem Name]{
   \ifx{#1}{Undefined Theorem Name}\renewcommand\nm@thmtype{theorem*}
   \else\renewcommand\thmsname{#1}\renewcommand\nm@thmtype{namedtheorem}
   \fi
   \begin{\nm@thmtype}}
   {\end{\nm@thmtype}}
  \theoremstyle{remark}
  \newtheorem{rem}[thm]{\protect\remarkname}
  \theoremstyle{plain}
  \newtheorem{lem}[thm]{\protect\lemmaname}
  \theoremstyle{definition}
  \newtheorem{example}[thm]{\protect\examplename}
\usepackage{soul}

\makeatother

  \providecommand{\corollaryname}{Corollary}
  \providecommand{\definitionname}{Definition}
  \providecommand{\examplename}{Example}
  \providecommand{\lemmaname}{Lemma}
  \providecommand{\propositionname}{Proposition}
  \providecommand{\remarkname}{Remark}
  \providecommand{\theoremname}{Theorem}
\providecommand{\theoremname}{Theorem}

\DeclareMathOperator{\Var}{Var}

\begin{document}

\title[Algorithmic randomness and computable measure theory]{On the close interaction between algorithmic randomness and constructive/computable measure theory}

\author{Jason Rute}

\thanks{Started January 15, 2015. Last updated \today.}
\begin{abstract}
This is a survey of constructive and computable measure theory with
an emphasis on the close connections with algorithmic randomness.
We give a brief history of constructive measure theory from Brouwer
to the present, emphasizing how Schnorr randomness is the randomness
notion implicit in the work of Brouwer, Bishop, Demuth, and others.
We survey a number of recent results showing that classical almost everywhere
convergence theorems can be used to characterize many of the common
randomness notions including Schnorr randomness, computable randomness,
and Martin-Löf randomness. Last, we go into more detail about computable
measure theory, showing how all the major approaches are basically
equivalent (even though the definitions can vary greatly).
\end{abstract}

\maketitle

\section{Introduction\label{sec:Introduction}}

Starting with the work of Turing in 1936 on the computability of
real numbers, it has been understood that many of the basic concepts
of analysis \textemdash{} e.g.~continuous functions, metric spaces,
and open sets \textemdash{} have computable analogues. This ``computable
interpretation'' of analysis has been developed through many interrelated
mathematical traditions, including the Russian and American constructivist
traditions, computable analysis, and reverse mathematics. 

One sub-branch of analysis, measure theory, has presented one of the
largest challenges to this program, as the American constructivist
Bishop observed.
\begin{quotation}
Any constructive approach to mathematics will find a crucial test
in the ability to assimilate the intricate body of mathematical thought
called measure theory. {[}...{]} It was recognized by Lebesgue, Borel,
and other pioneers in abstract function theory that the mathematics
they were creating relied, in a way almost unique at the time, on
set-theoretic methods, leading to results whose constructive content
was problematical. \cite[p.~154]{Bishop:1967lq}
\end{quotation}
In the 1960s, Bishop \cite{Bishop:1967lq} \textemdash{} in addition
to the Russian school of constructivists Šanin \cite{Sanin:1968dq}, Kosovski\u{\i}
\cite{Kosovskii:1969,Kosovskii:1969a,Kosovskii:1969b}, and Demuth
\cite{Demuth:1965,Demuth:1967a,Demuth:1967b,Demuth:1968,Demuth:1968b,Demuth:1969a,Demuth:1969b,Demuth:1969c,Demuth:1970a,Demuth:1973}
\textemdash{} overcame these difficulties to develop constructive
theories of measurable sets, measurable functions, integrable functions,
null sets, and almost everywhere convergence (drawing on earlier work
of Brouwer \cite{Brouwer:1919a}). Their work was later incorporated
into computable analysis and reverse mathematics.

Also in the 1960s, Martin-Löf \cite{Martin-Lof:1966mz} developed
his own notion of constructive null set, providing one of the most
successful definitions of randomness. Namely, a point is (Martin-Löf)
random if it is not in any (Martin-Löf) constructive null set. Around
1970, Schnorr \cite{Schnorr:1969ly,Schnorr1971a}
felt that Martin-Löf's notion of constructive null set was too inclusive.
He developed two other randomness notions, now known as Schnorr randomness
and computable randomness, each having their own corresponding notion
of constructive null set.

This article will show there is a deep connection between computable
measure theory and algorithmic randomness. At the heart of this discussion
is the notion of an effective (or constructive) null set.

After a short introduction to algorithmic randomness in Section~\ref{sec:Intro-to-rand},
we will give a survey of constructive measure theory in Section~\ref{sec:Constructive-survey}.
The purpose of this survey is twofold: to highlight common approaches
to measure theory among constructivists such as Brouwer, Demuth, Bishop,
Martin-Löf and others, and to show that algorithmic randomness naturally
arises out of these approaches. This will provide motivation for some
of the more technical results in the rest of the paper.

In Section~\ref{sec:Characterizing-randomness}, we survey a number
of recent results characterizing Schnorr randomness, computable randomness,
and Martin-Löf randomness using theorems from classical analysis.
For example, we will see that a real $x\in[0,1]$ is Martin-Löf random
if and only if $f$ is differentiable at $x$ for every computable
function $f\colon[0,1]\rightarrow\mathbb{R}$ of bounded variation.
Theorems of this type provide a useful-but-informal measure of the
``naturalness'' of a randomness notion. We also show how the characterization
results are connected to results in constructive analysis and reverse
mathematics. For example, the results that are constructively provable
(or provable in $\mathsf{RCA}_{0}$) are those most connected to Schnorr
randomness.

Lastly, in Section~\ref{sec:Foundations}, we turn to the foundations
of computable measure theory. We systematically organize the various
definitions in the computable and constructive mathematics literature of
effectively measurable set, measurable function, integrable function,
and almost uniform convergence. Although there is a number of definitions of these notions, 
they are basically equivalent. Once again randomness arises naturally.

Some \cite{Dasgupta:2011,Delahaye:2011} have argued that Martin-Löf
randomness is \emph{the correct randomness notion} \textemdash{}
just as Church-Turing computability is \emph{the correct computability notion}.
Others, have argued the same for other different randomness notions.
Porter \cite{Porter:2016}, on the other hand, has argued against any
one correct randomness notion. This survey, especially Section~\ref{sec:Characterizing-randomness},
supports this latter viewpoint. A variety of randomness notions have
been naturally characterized by a.e.~convergence theorems in analysis.

Nonetheless, there is one randomness notion that stands out in this
survey, especially given its limited treatment in the literature.
We will repeatedly see that Schnorr randomness, while much weaker than Martin-Löf
randomness, has very strong connections to constructive and computable
measure theory.

We hope this paper serves as a talking point between those from the
constructive analysis, the computable analysis, and the algorithmic
randomness communities. We also hope that others, who may not be interested
in randomness for its own sake, will still find this survey to be
a good starting point to learn about past and recent developments
in constructive and computable measure theory.

We are indebted to the editors Christopher Porter and Johanna Franklin.
Without their encouragement, this survey would never have been finished.
Moreover, they very graciously helped the author with a large amount of 
editing and reference tracking.

\section{A quick introduction to effective null sets and algorithmic randomness\label{sec:Intro-to-rand}}

Before getting into constructive measure theory in more depth, let us
introduce the concept of an effective null set. This notion is at
the heart of computable and constructive measure theory (especially
from a point-set view), and it is the starting point of algorithmic
randomness.

\subsection{Computable analysis on $\mathbb{R}$\label{subsec:Comp-analysis-on-R}}

We assume the reader has some basic understanding of what it means
for a function $f\colon\mathbb{N}\rightarrow\mathbb{N}$ to be computable.
See, for example, \cite{cooper04, soare16,o1,o2}. A real number $r$ is \emph{computable}
if there is a computable function $f\colon\mathbb{N}\rightarrow\mathbb{Q}$
such that $|f(n)-r|\leq2^{-n}$ for all $n\in\mathbb{N}$. We will
denote the set of computable reals as $\mathbb{R}_{\textnormal{comp}}$.

An \emph{effectively open set} $U\subseteq[0,1]$ is a set of the
form $\bigcup_{n}I_{n}$ where $(I_{n})_{n\in\mathbb{N}}$ is a computable
listing of open intervals with rational endpoints. (Under the usual
topology of $[0,1]$, the interval $[0,1/2)$ is an open interval
since it is the intersection of the open interval $(-1/2,1/2)$ and
$[0,1]$.) An \emph{effectively closed set} is the complement of
an effectively open set. If $D\subseteq[0,1]$, a \emph{computable function}
$f\colon D\rightarrow\mathbb{R}$ is a function such that for every
effectively open set $U\subseteq[0,1]$, we can (uniformly in the
code for $U$) compute an effectively open set $V$ such that $V\cap D=f^{-1}(U)\cap D$.
(This is one of many equivalent definitions.) If $f\colon\mathbb{R}_{\textnormal{comp}}\rightarrow\mathbb{R}_{\textnormal{comp}}$
is computable, then we say that $f$ is \emph{Markov computable}.

These definitions also extend naturally to Cantor space $\{0,1\}^\mathbb{N}$,
the space of infinite binary sequences.  Let $\{0,1\}^{<\mathbb{N}}$ denote the 
space of finite binary sequences.  Instead of rationals, use sequences 
containing finitely many $1$s. Instead of rational intervals, use cylinder sets 
$[\sigma]$ which is the set of all $x \in \{0,1\}^\mathbb{N}$ of which 
$\sigma \in \{0,1\}^{<\mathbb{N}}$ is a prefix.
For more background on computable analysis, see \cite{BrattkaWeihrauch99, BravermanCook06,Grzegorczyk57,Lacombe55a,Lacombe55b,Pour-El.Richards:1989,Weihrauch00}.

\subsection{Martin-Löf randomness and Schnorr randomness\label{subsec:MLR-and-SR}}

For now, let $\mu$ be the usual Lebesgue measure on $[0,1]$ or the 
\emph{fair-coin measure} on $\{0,1\}^\mathbb{N}$ given 
by $\mu([\sigma]) = 2^{-|\sigma|}$.
While there are many notions of algorithmic randomness and effective
null set, we start with the most important two. The first is due to Martin-Löf
\cite{Martin-Lof:1966mz}.
\begin{defn}
A \emph{Martin-Löf test} is a computable sequence of effectively
open sets $U_{n}$ such that $\mu(U_{n})\leq2^{-n}$ for all $n\in\mathbb{N}$. 
A \emph{Martin-Löf null set}
$E$ is any set covered by this test, that is $E\subseteq\bigcap_{n}U_{n}$.
A point $x$ is called \emph{Martin-Löf random} if it is not in any
Martin-Löf null set.
\end{defn}
The second definition of effective null set has roots in the constructive
measure theory of Brouwer, but was first introduced in a computability
theory setting by Schnorr \cite{Schnorr:1969ly,Schnorr1971a}.
\begin{defn}
A \emph{Schnorr test} is a computable sequence of effectively open
sets $U_{n}$ such that $\mu(U_{n})\leq2^{-n}$ for all $n\in\mathbb{N}$ 
and $\mu(U_{n})$
is computable uniformly in $n$. A \emph{Schnorr null set} $E$ is
any set covered by this test, that is $E\subseteq\bigcap_{n}U_{n}$.
A point $x$ is called \emph{Schnorr random} if it is not in any
Schnorr null set.
\end{defn}
Both of these definitions are effectivizations of outer regularity,
the result that any null set can be covered by an arbitrarily small
open set. 

By definition, every Schnorr null set is a Martin-Löf null
set. Therefore, every Martin-Löf random is Schnorr random. Moreover,
every computable real number $r$ is covered by a Schnorr null test.
For example, $1/2$ is covered by the Schnorr test $U_{n}=(1/2-2^{-(n+1)},1/2+2^{-(n+1)})$.
Therefore, no computable real is Schnorr random or Martin-Löf random.
However, consider the set $\mathbb{R}_{\textnormal{comp}}\cap[0,1]$
of all computable reals in the unit interval. This is where Schnorr null
sets and Martin-Löf null sets differ.
\begin{prop}[Martin-Löf \cite{Martin-Lof:1966mz}]
 There is a universal Martin-Löf null set which contains all other
Martin-Löf null sets. Therefore, $\mathbb{R}_{\textnormal{comp}}\cap[0,1]$
is a Martin-Löf null set.
\end{prop}

\begin{prop}[Schnorr \cite{Schnorr:1970a,Schnorr:1971a}]
\label{prop:Schnorr-null-contains-comp-pt-1} Given (a code for)
a Schnorr null set $E$, one can compute (uniformly in the code) a
computable point $x\notin E$. Therefore, $\mathbb{R}_{\textnormal{comp}}\cap[0,1]$
is not a Schnorr null set.
\end{prop}
This distinction has led many to assume Martin-Löf randomness is
more natural\footnote{``Despite Schnorr\textquoteright s critique, {[}Martin-Löf randomness{]}
has remained the paradigmatic notion of algorithmic randomness, and
has received considerably more attention than Schnorr randomness.
One reason may simply be that Martin-Löf\textquoteright s definition
came first, and is perfectly adequate for many results. Another important
reason, however, is that the mathematical theory of Schnorr randomness
is not as well behaved as that of {[}Martin-Löf randomness{]}. For
example, the existence of universal Martin-Löf tests (and corresponding
universal objects such as universal c.e. martingales and prefix-free
complexity) is a powerful tool in the study of {[}Martin-Löf randomness{]}
that is not available in the case of Schnorr randomness.'' \cite[\S7.1.2]{Downey.Hirschfeldt:2010}}, but note that Proposition~\ref{prop:Schnorr-null-contains-comp-pt-1}
is an effectivization of what is arguably the most fundamental principle
in point-set measure theory.
\begin{prop}
Any property which holds almost everywhere, holds somewhere.
\end{prop}

To see the connection, say that a property $P$ holds 
\emph{effectively almost everywhere} (in the sense of Schnorr) 
if the set of points not satisfying $P$ form a Schnorr null set.
Proposition~\ref{prop:Schnorr-null-contains-comp-pt-1}
says for every effectively such property $P$ we can effectively 
compute some $x$ for which the property $P$ holds.

\subsection{Other algorithmic randomness notions\label{subsec:Other-randomness-notions}}

Besides Schnorr and Martin-Löf randomness, there is a whole zoo of
randomness notions. We will need some of them at certain points, and
we list them here for reference. For more information the reader is
directed to the survey \cite{Downey.Hirschfeldt.Nies.ea:2006} or
the books by Downey-Hirschfeldt \cite{Downey.Hirschfeldt:2010} and
Nies \cite{Nies:2009}. (The reader may wish to skip this subsection
and refer back to it as needed.) 

The third most important randomness concept we will need is 
computable randomness. Also defined by Schnorr \cite{Schnorr1971a}, it arises
naturally in certain convergence theorems in analysis.  We use an equivalent 
definition due to Merkle, Mihailovi\'c, and Slaman \cite{Merkle.Mihailovic.Slaman:2006}.
A \emph{computable probability measure} on $[0,1]$
is a Borel probability measure $\nu$ on $[0,1]$ such that $p\mapsto\int_{0}^{1}p(x)\,d\mu(x)$
is a computable map from polynomials $p$ with rational coefficients
to their integrals.  (For Cantor space, $\{0,1\}^\mathbb{N}$, 
a \emph{computable probability measure} is a Borel probability measure on $\{0,1\}^\mathbb{N}$
for which $\sigma \mapsto \nu([\sigma])$ is computable for $\sigma \in \{0,1\}^{<\mathbb{N}}$.  
See Subsection~\ref{subsec:Computable-measures} for a uniform definition.)
\begin{defn}
A \emph{bounded Martin-Löf test} is a computable sequence of effectively
open sets $U_{n}$ such that there is a computable probability measure $\nu$
for which $\mu(U_{n} \cap A)\leq2^{-n}\nu(U_n \cap A)$ for 
any measurable set $A$. 
(It suffices that $A$ ranges over rational intervals $[a,b]$ for $[0,1]$ and  
cylinder sets $[\sigma]$ for $\{0,1\}^\mathbb{N}$.)
A \emph{computably null set}
$E$ is any set covered by this test, that is $E\subseteq\bigcap_{n}U_{n}$.
A point $x$ is called \emph{computably random} if it is not in any
computably null set.
\end{defn}

The remainder of the randomness notions are defined via complexity
of sets. The $\Pi_{1}^{0}$ and $\Sigma_{1}^{0}$ sets are respectively
the effectively closed and effectively open sets. A $\Sigma_{2}^{0}$
set, also known as an \emph{effective $F_{\sigma}$ set}, is a computable
union of $\Pi_{1}^{0}$ sets. Similarly, a $\Pi_{2}^{0}$ set, also
known as an \emph{effective $G_{\delta}$ set}, is a computable intersection
of $\Sigma_{1}^{0}$ sets. By recursion, one can define $\Sigma_{n}^{0}$
and $\Pi_{n}^{0}$ for all $n$.
\begin{defn}
A \emph{weak $n$-null set} is any subset of a null $\Sigma_{n+1}^{0}$
set. A point $x$ is called \emph{weak $n$-random} if it is not in
any weak $n$-null set (or equivalently is not in any null
$\Pi_{n}^{0}$ set). Weak 1-randomness is known as \emph{Kurtz randomness}.
\end{defn}
Many do not consider Kurtz randomness to be a true randomness notion.
One reason is that there is a Kurtz random real $x\in[0,1]$ whose
binary digits $(x_{n})$ do not satisfy the strong law of large numbers,
$\lim_{n}\frac{1}{n}\sum_{k=0}^{n-1}x_k=\frac{1}{2}$ \cite[3.5.3, 3.5.4]{Nies:2009}. 
\begin{defn}
\label{def:n-random}An \emph{$n$-Martin-Löf test} is a computable
sequence of $\Sigma_{n}^{0}$ sets $A_{n}$ such that $\mu(A_{n})\leq2^{-n}$.
An \emph{$n$-Martin-Löf null set} $E$ is any set covered by this
test, that is $E\subseteq\bigcap_{n}A_{n}$. A point $x$ is called
\emph{$n$-random} if it is not in any $n$-Martin-Löf null set.
\end{defn}
Notice  that $1$-randomness is Martin-Löf randomness. It also turns out
that $2$-randomness is equivalent to Martin-Löf randomness relative
to the halting problem $\emptyset'$, and $n$-randomness is equivalent
to Martin-Löf randomness relative to $\emptyset^{(n-1)}$.

In summary, the randomness notions are as follows listed in order of
strength (the weakest notions, which give rise to the largest set
of randoms, are listed first): Kurtz random, Schnorr random, computable
random, Martin-Löf random, weak $n$-random ($n\geq2$), $n$-random, 
weak $(n+1)$-random, ...

\section{Randomness and Constructive mathematics\label{sec:Constructive-survey}}

Constructive mathematics arose out of the desire to ensure that proofs
have computational meaning. While the early constructivist work of
Brouwer and others predates Turing's work, it is largely recognized
that constructivism has a computational interpretation (the Brouwer-Heyting-Kolmogorov interpretation). A constructive
proof of ``there exists a function $f$ ...,'' provides a construction of a computable function $f$.\footnote{This computable interpretation can be formalized via realizability or Hyland's effective topos.}

A consequence of this computable interpretation is that constructive
mathematics is consistent with \emph{Church's thesis}: all functions
are computable functions, and in particular, all reals are computable
reals.\footnote{Church's thesis in constructive mathematics
is stronger than the similarly named Church-Turing thesis (also called
Church's thesis), which only says that all intuitively computable
functions are computable in the sense of Church and Turing.  See the discussion in Beeson \cite[III.8]{Beeson:1985}.  (Although our version is closer to what Beeson calls the False Church's thesis.)} %
Nonetheless, it is still constructively provable, using Cantor's
diagonalization argument, that the set of real numbers is not countable.\footnote{To say that $[0,1]$ is not countable is to say there does not exist
an enumeration $\{r_{n}\}_{n\in\mathbb{N}}$ of $[0,1]$. Under the
computable interpretation this is saying that there is no computable
enumeration of $\mathbb{R}_{\textnormal{comp}}\cap[0,1]$.} On the other hand, it is more subtle to constructively prove that
the unit interval is not a null set. This comes down to the definition
of a null set. Classically, a set $A \subseteq [0,1]$ is null if for any $\varepsilon>0$,
the set $A$ can be covered by a sequence of intervals $(I_{n})_{n\in\mathbb{N}}$
such that the sum of the lengths of the intervals $\sum_{n\in\mathbb{N}}|I_{n}|$
is less than $\varepsilon$. Under the computable interpretation,
this covering corresponds to a Martin-Löf test.\footnote{That is, to constructively prove that a specific set $A$ is null,
we would for each (code of) $\varepsilon>0$, explicitly construct
a cover $(I_{n}^{\varepsilon})_{n\in\mathbb{N}}$ such that $\sum_{n\in\mathbb{N}}|I_{n}^{\varepsilon}|\leq\varepsilon$.
Letting $\varepsilon=2^{-k}$, we have that $U_{k}=\bigcup_{n}I_{n}^{\varepsilon}$
is an effectively open set uniformly in $k$ and that $\mu(U_{k})\leq\varepsilon=2^{-k}$
for all $k$.} However, Kreisel and Lacombe \cite{Kreisel.Lacombe:1957} and Zaslavski\u{\i}
and Ce\u{\i}tin \cite{Zaslavskii.Ceitin:1962} explicitly constructed
coverings of the computable reals which have arbitrarily small size.
Zaslavski\u{\i} and Ce\u{\i}tin call these \emph{singular coverings}.
Therefore, one quickly runs into the following \emph{paradox of singular coverings}\footnote{See Beeson \cite{Beeson:2005} for a more in-depth discussion on this
paradox, including a work-around not mentioned here.}:
\begin{thm}[Paradox of singular coverings, first version] The following set of statements is constructively inconsistent for any definition of ``null set.''
\begin{enumerate}
\item The set of computable reals $\mathbb{R}_{\textnormal{comp}} \cap [0,1]$ is a
null set.
\item (Church's thesis) All reals are computable.  (Hence $[0,1]\subseteq\mathbb{R}_{\textnormal{comp}}$.)
\item If $A \subseteq B$ and $B$ is null, then so is $A$.
\item The unit interval $[0,1]$ is not null.
\end{enumerate}
\end{thm}
For, (1)--(3) imply the negation of (4).  Statements (3) and (4) are basic facts of measure theory that one needs to develop a consistent notion of measurable set and measure.  That means in order to develop measure theory, we need to reject (1) or (2).  Some, for example Martin-Löf, have used this argument to reject Church's thesis.  The negation of Church's thesis, not all reals are computable, does not actually imply (constructively) that there is a noncomputable real.  For example, Brouwer's intuitionism---in particular his fan principle---is incompatible with Church's thesis, but still compatible with \emph{weak Church's thesis}: there does not exist a nonconstructive real.

Nonetheless, there are still issues with adopting the above ``covering'' definition of a null set.
\begin{thm}[Paradox of singular coverings, second version] The following set of statements is constructively inconsistent for any definition of ``null set.''
\begin{enumerate}
\item The unit interval $[0,1]$ is a measurable set with measure one.
\item The set of computable reals $\mathbb{R}_{\textnormal{comp}} \cap [0,1]$ is a null set.
\item If $A$ has positive measure and $B$ is null then $A\smallsetminus B$ has positive measure.
\item Every measure one set contains a point.
\item (Weak Church's thesis) There does not exist a noncomputable real.
\end{enumerate}
\end{thm}
For, (1)--(4) imply the existence of a noncomputable real, contradicting (5).  Again, (1), (3), and (4) are basic properties of measure theory that would be nice to have in any constructive development of point-set measure theory.  Again, one is left with the choice of denying weak Church's thesis or using a different definition of null set in which one can't constructively prove that the real numbers are null.  Most constructivists, starting with Brouwer, opted to go with the latter, defining null sets via \emph{regular coverings}, that is
coverings where $\sum_{n}|I_{n}|$ constructively exists.\footnote{In constructive mathematics, one cannot in general prove that a bounded monotone sequence converges. There are examples of bounded monotone computable sequences whose limit is not computable.}

Regular coverings, under the computable interpretation, correspond
to Schnorr tests. Indeed, Schnorr \cite{Schnorr:1969ly,Schnorr:1971a}
referred to his null sets as ``total recursive null sets in the sense
of Brouwer''.

What follows is a short survey on constructive measure theory and
related subjects, emphasizing the deep connections with effective
null sets and, in some cases, algorithmic randomness.

\subsection{Brouwerian intuitionism\label{subsec:Brouwer}}

In 1919, Brouwer \cite{Brouwer:1919a} developed a constructive measure
theory on the unit square. (See the presentation in Heyting's book
\cite[Ch.~VI]{Heyting:1956}.) In Brouwer's measure theory, a set
is null if it is enclosed in a measurable open set of arbitrarily
small measure. Here a \emph{measurable open set} is an open set in
which the measure constructively exists, and \emph{arbitrarily small}
means that given a natural number $n$, one can construct a measurable
open set enclosing $A$ with measure less than $2^{-n}$. In the computable
interpretation, a measurable open set corresponds to an effectively
open set of computable measure, and therefore the Brouwerian null
sets correspond to Schnorr null sets.

Further, in Brouwer's measure theory, a set $Q$ is \emph{measurable}
if for each $n$, there is a measurable open set $U_{n}$ of measure
less than $2^{-n}$ and a finite union of rational rectangles $V_{n}$
such that $Q=V_{n}$ outside of $U_{n}$ (that is $Q\triangle V_{n}\subseteq U_n$ 
where $\triangle$ is symmetric difference) \cite[p.~29]{Brouwer:1919a}\cite[\S\S6.3.1,Thm.~1]{Heyting:1956}.
Then $\mu(Q)$ is defined as $\lim_n \mu(V_n)$, where the measure
$\mu(V_n)$ is the geometric area of $V_n$.
Brouwer gave definitions of measurable functions and integrable functions
as well. In general, Brouwer's approach is the one followed by many later
constructivists, insofar as their approaches are equivalent.\footnote{One slight difference with later constructivists is that in Brouwer's
measure theory, a measurable function need not be defined on a set
of full measure. In this case the function is assumed to be zero on
almost all of those undefined points. However, it is shown that such
partial functions can be extended to a full domain \cite[\S\S6.2.2]{Heyting:1956}.
In that case, Brouwer's definition is compatible with the later constructivists.}

Brouwer and his students developed a large amount
of measure theory constructively, including fundamental results about
measurable functions and sets, the monotone convergence theorem, the
dominated convergence theorem, and Egoroff's theorem \cite[Ch.~VI]{Heyting:1956}.
However, it should be noted that Brouwer's intuitionism is incompatible
with classical logic. For example, it is a Brouwerian theorem that
every function on the unit interval is uniformly continuous. As a
corollary, every bounded function defined almost everywhere is measurable 
\cite[\S\S6.2.2, Thm.~1]{Heyting:1956}.
Also, Brouwer adopted the fan principle, which later constructivists
deemed nonconstructive. Using this theorem, one can prove the dominated
convergence theorem and Egoroff's theorem \cite[\S\S6.5.4]{Heyting:1956}. The latter says that (on
a probability space) a.e.~convergence implies almost uniform convergence.

\subsection{The Russian school of constructive mathematics\label{subsec:Russians}}

The Russian school of constructive mathematics \textemdash{} led
by Markov and his students Šanin, Zaslavski\u{\i}, and Ce\u{\i}tin
\textemdash{} combined the ideas of Turing and Brouwer. In particular
Church's thesis \textemdash{} that every function is (Markov) computable
\textemdash{} was explicitly assumed. Therefore, Russian recursive
constructivism is very similar to modern computable analysis (except
that the Russian constructivists avoided most nonconstructive principles
such as the law of the excluded middle\footnote{They did however adopt Markov's principle,
which states that for each binary sequence $(a_n)$, if every no term $a_n$ equals $0$, 
then there exists a term equal to $1$.  This is a weak form of the law of excluded middle.}, and avoided reference to non-computable object.) See the surveys \cite{kushner99,Demuth.Kucera:1979} and the books \cite{kushner84,bridgesrichman87} for more on Russian constructive mathematics.

In 1962, Šanin wrote a
book on constructive analysis, emphasizing constructive metric spaces, which appeared in English translation in 1968 \cite{Sanin:1968dq}.
Formally, a constructive metric space is identified with a metric
$\rho$ on the natural numbers, and the constructive points in this
constructive metric space are identified with constructive sequences
$(n_{k})$ of natural numbers such that $\rho(n_{k},n_{\ell})\leq2^{-k}$
for all $k\leq\ell$. The idea is to encode a metric on a countable
set, e.g.\ the Euclidean distance on $\mathbb{Q}$, and the constructive
metric space is the completion of this metric, e.g.\ $\mathbb{R}$.
(However, by Šanin's use of Church's thesis, this constructive completion
only consists of computable points.)

Šanin used computable metric spaces to give constructive definitions
of measurable sets, measurable functions, and integrable functions.
For example, consider the $L^{1}$-metric $\rho(f,g)=\int_{0}^{1}\left|f(x)-g(x)\right|dx$
on rational step functions. This describes a constructive metric space,
and the corresponding constructive points are the constructive integrable
functions \textemdash{} the \emph{integrable FR-constructs} in Šanin's
terminology. Similarly, Šanin defined measurable sets and measurable
functions in a similar manner (see Subsection~\ref{subsec:Point-free-approach}).
Kosovski\u{\i} \cite{Kosovskii:1969,Kosovskii:1969a,Kosovskii:1969b,Kosovskii:1970,Kosovskii:1973,Kosovskii:1973a}
further extended Šanin's work to probability theory, proving constructive
versions of the strong law of large numbers, developing a theory of
constructive stochastic processes, and extending Šanin's ideas to
arbitrary spaces given by normed Boolean algebras of sets.

Šanin's and Kosovski\u{\i}'s approach is different from Brouwer's
in that it is \emph{point-free}. Each integrable FR-construct is
not a true function, but instead a point in a metric space of function-like
objects. (Recall that, classically, the metric space $L^{1}([0,1])$
is the space of \emph{equivalence classes} of integrable functions
modulo a.e.~equivalence.) Unlike Brouwer's integrable functions,
the statement $f(0)=1$ is not meaningful for an integrable FR-construct
$f$. We will return to this point-free theme in Subsections~\ref{subsec:Point-free}
and \ref{subsec:Pointwise-approaches}.

Also in 1962, Zaslavski\u{\i} and Ce\u{\i}tin \cite{Zaslavskii.Ceitin:1962}
wrote about the singular coverings mentioned at the beginning of this
section. While their focus was on the pathological case of singular
coverings, they added the following note.
\begin{quote}
We call a covering $\Phi$ \emph{regular} if the sequence of
numbers $\sum_{k=0}^{n}|\Phi_{k}|$ is constructively convergent as
$n\rightarrow\infty$. The set $\mathcal{M}$ of {[}constructive real
numbers{]} will be said to be a \emph{set of measure zero} if for
arbitrary $\varepsilon$ there can be realized a regular $\varepsilon$-bounded
covering by intervals of the set. {[}...{]} Consequently, in spite
of the existence of constructive singular coverings, it is possible
to give a reasonable definition of the constructive concept of a set
of measure zero. Other concepts of the constructive theory of measure
can be defined in a similar way. \cite[p.~58 in English translation]{Zaslavskii.Ceitin:1962}
(Emphasis in original.)
\end{quote}
While Zaslavski\u{\i} and Ce\u{\i}tin do not define such ``other
concepts'', Demuth \cite{Demuth:1965,Demuth:1967a,Demuth:1967b,Demuth:1968,Demuth:1968b,Demuth:1969a,Demuth:1969b,Demuth:1969c,Demuth:1970a,Demuth:1973}
does take up this work, giving constructive definitions of integrable
functions and measurable sets.
(A detailed survey of Demuth's work on constructive measure theory
can be found in Demuth and Ku\v{c}era \cite{Demuth.Kucera:1979}.
Also see the surveys by Slaman and Ku\v{c}era \cite[Rmk~3.5]{Kucera.Slaman:2001}
and Ku\v{c}era, Nies, and Porter \cite{Kucera.Nies.Porter:}.) Demuth's
work is particularly relevant because he, independently of Martin-Löf
and Schnorr, defined the same randomness notions (or at least considered
the corresponding null sets). Ku\v{c}era, Nies, and Porter comment
on Demuth's path to randomness.
\begin{quote}
Demuth considered a number of different notions of effective null
set. They are equivalent to several major randomness notions that
have been introduced independently. 

It is striking that Demuth never actually referred to random or non-random
sequences. Instead, he characterized these classes in terms of non-approximability
in measure and approximability in measure, respectively. This reflects
the fact that Demuth\textquoteright s motivation in introducing these
classes differed significantly from the motivation of the recognized
\textquotedblleft fathers\textquotedblright{} of algorithmic randomness.
Whereas the various randomness notions were introduced and developed
by Martin-Löf, Kolmogorov, Levin, Schnorr, Chaitin, and others in
the context of classical probability, statistics, and information
theory, Demuth developed these notions in the context of and for application
in constructive analysis, where the notion of approximability plays
a central role \cite[\S4]{Kucera.Nies.Porter:}.
\end{quote}
Demuth's measure theory takes place entirely on the constructive reals.
A property $P$ of the constructive real numbers is said to hold for
\emph{almost every constructive real number} if (in modern terminology)
it holds outside of a Schnorr null set \cite[p.~87]{Demuth.Kucera:1979}.
Demuth gave a set-point interpretation of Šanin's point-free approach
(see the remark in \cite{Demuth:1968}) as follows.  A partial function $f\colon{\subseteq{}}\mathbb{R}_{\textnormal{comp}}\rightarrow\mathbb{R}_{\textnormal{comp}}$
is \emph{integrable} if there is a computable sequence of rational
step functions $s_{n}$ such that for all $n\geq m$, $\|s_{m}-s_{n}\|_{L^{1}}\leq2^{-m}$
and $f(x)=\lim_{n}s_{n}(x)$ for almost every $x\in\mathbb{R}_{\textnormal{comp}}$.
The integral $\int_{0}^{1}f(x)\,dx$ is equal to $\lim_{n}\int_{0}^{1}s_{n}(x)\,dx$
(where the integral of the step function $s_{n}$ is defined in the
usual way). Demuth similarly defines a \emph{measurable} function
using the metric $\rho(f,g)=\int_{0}^{1}\frac{|f(x)-g(x)|}{1+|f(x)-g(x)|}\,dx$.
A set $A\subseteq\mathbb{R}_{\textnormal{comp}}$ is \emph{measurable}
if there is an integrable function $f\colon{\subseteq{}}\mathbb{R}_{\textnormal{comp}}\rightarrow\mathbb{R}_{\textnormal{comp}}$
such that $\mathbf{1}_{A}(x)=f(x)$ for almost every\ $x\in\mathbb{R}_{\textnormal{comp}}$.
Then $\mu(A)$ is defined as $\int_{0}^{1}f(x)\,dx$ \cite[\S4]{Demuth.Kucera:1979}. 

While Demuth's measurable sets are restricted to the constructive
real numbers, this is just the computable interpretation of constructive mathematics at play.
His definitions work equally well on the whole unit interval, and if taken as such, they are 
constructively equivalent to those of Brouwer.%
\footnote{When Demuth considers an ``integrable function''  
$f\colon{\subseteq{}}\mathbb{R}_{\textnormal{comp}}\to\mathbb{R}_{\textnormal{comp}}$ he is defining $f$
as a constructive limit of rational ``step functions'' $s_n$.  
While these ``step functions'' are only defined on $\mathbb{R}_{\textnormal{comp}}$, they have natural extensions $\bar{s}_n$ defined on $[0,1]$.  The classical limit $\lim_n \bar{s}_n$ of these step functions converges almost everywhere to a function $\bar{f}\colon[0,1]\to \mathbb{R}$.  Then $f=\bar{f}\upharpoonright\mathbb{R}_{\textnormal{comp}}$ and
the classical integral of $\bar{f}$ is the same as Demuth's ``integral'' of $f$.  
Moreover, a ``measurable set'' $A \subseteq \mathbb{R}_{\textnormal{comp}}$ in Demuth's terminology can be identified with a $\{0,1\}$-valued ``integrable function'' $f$.  By extending $f$ to its classical counterpart $\bar{f}$, we get a set $\bar{A} = \{x : f(x) = 1\}$ such that $A = \bar{A} \cap \mathbb{R}_{\textnormal{comp}}$ for ``almost every constructive real'' $x$ in the sense of Demuth,
and Demuth's ``measure'' of $A$ is the same as the classical measure of $\bar{A}$.
}%

Demuth proved constructive versions of a number of differentiability results in measure
theory including the Lebesgue differentiation theorem \cite[Thm.~4.14]{Demuth.Kucera:1979}.
Demuth was particularly interested in the differentiability of functions
of bounded variation. He showed that for every constructively absolutely
continuous function $f\colon\mathbb{R}_{\textnormal{comp}}\rightarrow\mathbb{R}_{\textnormal{comp}}$,
the set of non-differentiable\footnote{Technically, this is a notion of ``non-pseudo-differentiability''
since Markov computable functions are only defined on constructive
reals. See \cite{Demuth.Kucera:1979} or \cite{Kucera.Nies.Porter:}
for more details.} points can be covered by a (not necessarily regular) constructive
covering. Translated into a modern perspective, Demuth's result shows
that absolutely continuous Markov computable functions are differentiable
at Martin-Löf randoms (cf.\ Theorem~\ref{thm:MLR-bdd-var}).
To avoid the paradox of singular coverings, Demuth (slightly) abandoned
Church's thesis, enlarging the constructive interval to contain ``pseudo-reals'',
that is reals computable in the halting problem, $\emptyset'$ (see, for instance, \cite{Demuth:1975} and  \cite{Demuth:1975zr}).

\subsection{Bishop's constructive mathematics\label{subsec:Bishop}}

In 1967, Bishop published a book on constructive mathematics \cite{Bishop:1967lq},
showing that a large amount of mathematical analysis could
be proved constructively. A major portion of his work was on measure
theory. Whereas Brouwer's intuitionism and the constructive mathematics of
the Russian school allows one to prove nonclassical results (such as
all functions are uniformly continuous or all functions are computable)
Bishop's constructivism is compatible with classical mathematics \cite{Beeson:1985}.
Therefore, any result proved in Bishop's book is classically valid,
but also constructive \textemdash{} and therefore has a computable
interpretation.

Bishop's measure theory progressed through a number of revisions.
His first development \cite[Ch.~6]{Bishop:1967lq} was for probability
measures on locally compact metric spaces. (See Bridges and Demuth
\cite{Bridges.Demuth:1991} or Beeson \cite[{\S}I.13]{Beeson:1985}\cite{Beeson:2005}
for short presentations.) Later Bishop and Cheng \cite{Bishop.Cheng:1972}
extended this framework to arbitrary integration spaces via the Daniell
integral. (Also see Bishop and Bridges \cite[Ch.~6]{Bishop.Bridges:1985}.) In both cases,
measures are defined via a linear integration functional. We will
briefly explain how Bishop's approach applies to the space $[0,1]$
with the Lebesgue measure $\mu$. This measure $\mu$ can be defined
via the Riemann integral $\int_{0}^{1}f\,(x)\,dx$ on uniformly continuous
functions $f\colon[0,1]\rightarrow\mathbb{R}$. An \emph{integrable function}
is a partial function $f\colon{\subseteq{}}[0,1]\rightarrow\mathbb{R}$
constructed as follows. Take a sequence of uniformly continuous functions
$f_{n}$ such that $\sum_{n}\int_{0}^{1}|f_{n}(x)|\,dx$ constructively
converges.  Set the domain of $f$ to be the set of all $x\in [0,1]$ 
such that $\sum_{n}|f_{n}(x)|$ constructively
converges. For such $x$, set $f(x)=\sum_{n}f_{n}(x)$.
A set is \emph{full} if it contains the domain of some integrable
$f$.

If the sequence $f_{n}$ is a computable sequence of uniformly continuous
functions, then the corresponding full set $\{x\colon\sum_{n}|f_{n}(x)|\ \text{converges}\}$
is the complement of a Schnorr null set. Conversely, every Schnorr
null set is of this form (Theorem~\ref{thm:SR-monotone}).
Moreover, Bishop's definitions and theorems largely agree with those
of Brouwer.\footnote{Unlike Brouwer, Bishop does not adopt the fan principle. Therefore,
he cannot prove Ergorov's theorem that almost everywhere convergence
is the same as almost uniform convergence. Instead his definition
of almost everywhere convergence is closer to almost uniform convergence.
In particular, his dominated convergence theorem is weaker than Brouwer's,
and therefore weaker than the classical version.} A noteworthy constructive theorem of Bishop is that every measurable
set of positive measure contains a point \cite[Ch.~6, Prop.~2]{Bishop:1967lq}
(compare with Proposition~\ref{prop:Schnorr-null-contains-comp-pt-1}).

Bishop-style constructivism continues to received a lot of attention.
There have been a number of results in Bishop-style constructive measure theory 
and probability theory \cite{Chan:1969,Chan:1972,Chan:1974,Chan:1974a,Chan:1975,Bridges:1977a,Bridges:1979,Chan:1981},
including on advanced topics such as 
ergodic theory \cite{Bishop:1967lq, Bishop:1967a, Nuber:1972, Spitters:2002, Spitters:2006, Spitters:2006b},
stochastic processes \cite{Chan:1972a,Chan:1976,Chan:1981},
potential theory \cite{Chan:1977,Chan:1981},
and quantum mechanics \cite{Hellman:1993,Hellman:1997,Bridges.Svozil:2000}.
It also influenced some of the later Russian constructivists, such as Kreinovich's 
work on constructive Wiener measure \cite{Kreinovich:1974a,Kreinovich:1974}.

\subsection{Martin Löf's constructive mathematics\label{subsec:Martin-Lof}}

In 1966, Martin-Löf \cite{Martin-Lof:1966mz} introduced his definition
of constructive null set and Martin-Löf randomness. Later, he turned
his focus to constructive type theory. In 1970, during this transitionary
period, Martin-Löf wrote a book on constructive analysis \cite{Martin-Lof:1970a},
including a chapter devoted to measure theory. 

His style is similar to that of the Russian school, mentioning computable
objects explicitly, but he does not work explicitly in the constructive
real numbers. Indeed, Martin-Löf rejects the idea that the continuum
is made up only of computable points. He invokes the existence of
singular coverings \textemdash{} which is a stronger form of Kreisel and Lacombe's
theorem \cite{Kreisel.Lacombe:1957} that there is an effective open set
not equal to the reals which contains all computable reals.  Of this result Martin-L\"of writes,
\begin{quote}
In classical mathematics the continuum is conceived as the totality
of its points. One might therefore, like Markov and his school, try
to constructivize the continuum by looking upon it as the totality
of its constructive points. This leads, as shown by Kreisel and Lacombe's theorem,
to a theory which is radically different from Brouwer's. \cite[p.~57]{Martin-Lof:1970a}
\end{quote}
Martin-Löf's definition of measurable set is as follows.
\begin{quote}
A Borel set $A$ is \emph{measurable} if for every computable real
number $\varepsilon>0$ {[}...{]} we can find a simple set $P$ {[}that
is, a finite union of disjoint basic open sets{]} and an open set
$U$ such that
\[
A\triangle P\subseteq U
\]
 and $U$ is bounded by $\varepsilon$ {[}that is, $\mu(Q)\leq\varepsilon$
for every simple set $Q\subseteq U${]}. \cite[p.~92]{Martin-Lof:1970a}
\end{quote}
Notice that unlike Brouwer's definition before, $\mu(U)$ need not
(constructively) exist. Martin-Löf was aware of the difference.
\begin{quote}
There are several reasons why we have chosen a more inclusive definition
of measurability than Brouwer did. First of all, the problem has always
been to find a consistent extension of the measure, first defined
for simple sets only, which goes as far as possible. Our extension,
although going further than Brouwer's entails no departure from the
constructive standpoint. \cite[p.~100]{Martin-Lof:1970a}
\end{quote}
He was also aware that this would lead to a singular covering of the
computable reals.
\begin{quote}
Secondly, the fact that our definition allows the construction of
an inner limit set of measure zero which contains all constructive
points, although troublesome to those whose continuum consists of
constructive points only, is in full agreement with the intuitionistic
concept of the continuum as a medium of free choice. \cite[p.~101]{Martin-Lof:1970a}
\end{quote}
Last, he ends his defense of his definition of measurable set by referring
to his notion of randomness and his theorem that there is a universal
Martin-Löf constructive null set.
\begin{quote}
Thirdly, the definition we have adopted enables us to prove a new
theorem which may serve as a justification of the notion of a random
sequence conceived by von Mises and elaborated by Wald and Church
1940. \cite[p.~101]{Martin-Lof:1970a}
\end{quote}

\subsection{Reverse mathematics\label{subsec:Reverse-math}}

Constructive mathematics gets its computable interpretation from restricting
itself to a subset of classical logic. There is, however, another
way of doing mathematics, which both has a computational interpretation
and uses classical logic. That is $\mathsf{RCA}_{0}$, a subsystem
of second order arithmetic, which forms the basis for the reverse
mathematics program of Friedman and Simpson \cite{Simpson:2009}.

While $\mathsf{RCA}_{0}$ and $\mathsf{BISH}$ (Bishop's constructive
system) are similar, there are also key differences. $\mathsf{RCA}_{0}$
uses classical logic, whereas $\mathsf{BISH}$ does not. Conversely,
various versions of the axiom of choice hold in $\mathsf{BISH}$ which
do not in $\mathsf{RCA}_{0}$. There are also differences in methodology
between reverse mathematics and Bishop style constructivism. While
a constructivist desires to move much of mathematics under a constructive
lens, the goal of reverse mathematics is to determine exactly which
set existence axioms (added to $\mathsf{RCA}_{0}$) are required to
prove a theorem of mathematics. It turns out that a large number of
theorems in mathematics are equivalent (over $\mathsf{RCA}_{0}$)
to one of the following five systems of reverse mathematics (listed
in increasing proof-theoretic strength), $\mathsf{RCA}_{0}$, $\mathsf{WKL}_{0}$,
$\mathsf{ACA}_{0}$, $\mathsf{ATR}_{0}$, and $\Pi_{1}^{1}\text{-}\mathsf{CA}_{0}$.
(For an introduction to reverse mathematics, see \cite{Simpson:2009}.)

However, when Yu and Simpson \cite{Yu.Simpson:1990} looked at the
reverse mathematics of measure theory, another system $\mathsf{WWKL}_{0}$
arose, strictly between $\mathsf{RCA}_{0}$ and $\mathsf{WKL}_{0}$.
The axiom weak weak König's lemma ($\mathsf{WWKL}$) states that if
$T$ is a subtree of $\{0,1\}^{<\mathbb{N}}$ with no infinite path,
then
\[
\lim_{n\rightarrow\infty}\frac{|\{\sigma\in T : |\sigma|=n\}|}{2^{n}}=0.
\]
The system $\mathsf{WWKL}_{0}$
is $\mathsf{RCA}_{0}+\mathsf{WWKL}$. Yu and Simpson \cite{Yu:1987ff,Yu.Simpson:1990,Yu:1990,Yu:1993,Yu:1994oz,Yu:1996}
showed that a large amount of measure theory can be developed in the
system $\mathsf{WWKL}_{0}$. Moreover, the axiom $\mathsf{WWKL}$
is equivalent over $\mathsf{RCA}_{0}$ to a number of basic principles
of measure theory (see \cite[{\S}X.1]{Simpson:2009}):
\begin{itemize}
\item Every closed set of positive measure contains a point.
\item Every sequence of intervals $(a_{n},b_{n})$ covering $[0,1]$ satisfies
$\sum_{n=0}^{\infty}(b_{n}-a_{n})\geq1$.
\item If $U,V\subseteq\{0,1\}^{\mathbb{N}}$ are disjoint open sets such
that $U\cup V=\{0,1\}^{\mathbb{N}}$ then $\mu(U)+\mu(V)=1$.
\end{itemize}
In short (using the terminology from earlier), $\mathsf{WWKL}$ prevents
the pathologies of singular coverings. $\mathsf{WWKL}$ is also closely
related to Martin-Löf randomness. Indeed $\mathsf{WWKL}$ is equivalent (over $\mathsf{RCA}_{0}$)
to the existence of a Martin-Löf random relative to each $x\in\{0,1\}^{\mathbb{N}}$
\cite[Thm~3.1]{Avigad.Dean.Rute:2012}.

The reverse mathematics of measure theory relies on both point-free
definitions of integrable functions and sets (using the $L^{1}$ metric
space), as well as pointwise versions. Yu \cite{Yu:1994oz}, Brown,
Giusto, and Simpson \cite{Brown.Giusto.Simpson:2002}, Simic \cite{Simic:2004}, and Avigad,
Dean, and Rute \cite{Avigad.Dean.Rute:2012}
define the pointwise version of an integrable function $f$ as the
pointwise limit of a sequence $(p_{n})$ of certain continuous functions which
approximate $f$ in the $L^{1}$-norm. Using $\mathsf{WWKL}_{0}$
they show that these $(p_{n})$ converge outside of a (relativized) Martin-Löf
null set.\footnote{There is some ambiguity in the definition of ``null set'' in this literature.
Yu \cite{Yu:1994oz} considers almost everywhere to mean outside a
``null $G_{\delta}$ set,'' that is, a set $G=\bigcap_{n}U_{n}$
where the sets $U_n$ are open and $\lim_{k}\mu(\bigcap_{n<k}U_{n})=0$
(more exactly, for all $\varepsilon>0$ there is some $n$ such that $\mu(\bigcap_{k<n} U_k) < \varepsilon$).  
In reverse mathematics, this would correspond to a null set for weak $2$-randomness.
This is likely a error, because later in the same paper she assumes the stronger property 
that $\mu(U_n) \leq 2^{-n}$.
This would correspond to a Martin-Löf null set. 
Brown, Giusto, and Simpson \cite{Brown.Giusto.Simpson:2002}
and Simic \cite{Simic:2004} both use the Martin-Löf random version.
Avigad, Dean, and Rute use null $G_{\delta}$ sets, but in the context
of the axiom $2\text{-}\mathsf{WWKL}$ where the differences are less
important.} By the later work of Pathak, Rojas, and Simpson \cite{Pathak.Rojas.Simpson:2014}
and Rute \cite{Rute:2013pd}, as well as the constructivists already
mentioned, it is likely provable in $\mathsf{RCA}_{0}$ that this
convergence happens outside of a (relativized) Schnorr null set. Indeed, it seems that
a large amount of measure theory can be developed in $\mathsf{RCA}_{0}$
\textemdash{} including many of the results proved using $\mathsf{WWKL}_{0}$
in Yu \cite{Yu:1987ff,Yu.Simpson:1990,Yu:1994oz}, Brown, Giusto,
and Simpson \cite{Brown.Giusto.Simpson:2002}, and Simic \cite{Simic:2004}.

Nonetheless, there are a number of theorems not provable in $\mathsf{RCA}_{0}$.
For example, over $\mathsf{RCA}_{0}$, both (a certain version of)
the monotone convergence theorem \cite{Yu:1994oz} and the Vitali
covering theorem \cite{Brown.Giusto.Simpson:2002} are equivalent
to $\mathsf{WWKL}$. Yu showed that Borel regularity is provable in
$\mathsf{ATR}_{0}$ \cite{Yu:1993}, and that many theorems of measure
theory are equivalent (over $\text{\ensuremath{\mathsf{RCA}_{0}}}$)
to $\mathsf{ACA}$ \cite{Yu:1987ff,Yu:1990,Yu:1996}. Simic \cite{Simic:2004, Simic:2007} 
showed that the pointwise ergodic theorem is equivalent to $\mathsf{ACA}$.  Avigad and Simic \cite{Avigad.Simic:2006} showed the same for the mean ergodic theorem.
Avigad, Dean, and Rute \cite{Avigad.Dean.Rute:2012} showed that the
following are all equivalent (over $\text{\ensuremath{\mathsf{RCA}_{0}}}$)
to an axiom called $2\text{-}\mathsf{WWKL}$:
\begin{itemize}
\item Egoroff's theorem
\item the Cauchy version of the dominated convergence theorem
\item every $G_{\delta}$ set of positive measure contains a point
\item collection axiom $\mathsf{B}\Sigma_{2}$ plus the existence of a $2$-random
(Definition~\ref{def:n-random}) relative to each $x\in\{0,1\}^{\mathbb{N}}$.
\end{itemize}

We also remark that reverse mathematics has inspired a similar program
called constructive reverse mathematics which replaces the base theory
$\mathsf{RCA}_{0}$ with $\mathsf{BISH}$ (or some other suitable
constructive base theory). Nemoto \cite{Nemoto:2010} has investigated
$\mathsf{WWKL}$ in constructive reverse mathematics, and Beeson \cite{Beeson:2005}
has investigated the constructive strength of the statement the every
sequence of intervals $(a_{n},b_{n})$ covering $[0,1]$ satisfies
$\sum_{n=0}^{\infty}(b_{n}-a_{n})\geq1$.

\subsection{Computable analysis\label{subsec:Computable-analysis}}

Computable analysis, like constructive analysis, studies the computable
content of theorems in mathematical analysis. Unlike constructive
mathematics or $\mathsf{RCA}_{0}$, computable analysis does not rely
on any restricted framework of logic or mathematics. Instead, it explicitly
refers to computable functions, computable reals, etc. Also like constructive
analysis, computable analysis developed in many separate but interrelated
traditions (see Avigad and Brattka \cite{Avigad.Brattka:2014} for
a historical survey). 

Early work combining the measure-theoretic and computability theoretic
can be found in Kreisel and Lacombe's \cite{Kreisel.Lacombe:1957}
result that there is a $\Sigma_{1}^{0}$ set of arbitrarily small
measure covering all the computable reals, as well as Jockusch and
Soare's \cite{Jockusch.Soare:1972} work showing that the complete
extensions of Peano arithmetic have measure zero.

Later Friedman and Ko \cite{Ko.Friedman:1982,Ko:1986,Ko:1991} studied
the polynomial-time complexity of measurable functions and sets, via
approximability. Ko \cite[Ch.~5]{Ko:1991} showed that by replacing
``polynomial-time computable'' with ``computable'', the approximable
sets and functions are equivalent to the measurable sets and functions
of Šanin. Pour-El and Richards \cite{Pour-El.Richards:1989} developed
computable analysis on Banach spaces, focusing significantly on $L^{p}$
spaces, again using a point-free treatment similar to Šanin.

Starting around the turn of the millennium, there have been a large
number of papers on computable measure theory. Many of these papers
follow the type-$2$ effectivity approach \cite{Weihrauch00,Brattka.Hertling.Weihrauch:2008} or the
domain theory approach \cite{Abramsky.Jung:1994}. Most of these papers have
been concerned with computable representations of measures or probability distributions
\cite{Weihrauch:1999,Muller:1999,Wu.Weihrauch:2006,Schroder.Simpson:2006,Schroder:2007kx,Edalat:2009nx,Hoyrup.Rojas:2009,Mori.Tsujii.Yasugi:2013,Collins:}.
While most of these representations are equivalent, the generality
of the underlying spaces vary. Other papers have been about computable
representations of measurable sets, integrable functions, and measurable
functions or their properties
\cite{Wu.Ding:2005,Wu.Ding:2006,Edalat:2009nx,Hoyrup.Rojas:2009a,Hoyrup.Rojas:2009b,Bosserhoff:2008,Wu:2012,Weihrauch.Tavana:2014,Weihrauch:2017,Collins:}. Again, these representations are basically
equivalent, but the details are a bit more complicated. As we will
see in Section~\ref{sec:Foundations}, the various representations
can be broken up into three categories corresponding to those that
are point-free, those that are defined outside of a Martin-Löf null
set, and those that are defined outside of a Schnorr null set. 

Yet others are interested in computable stochastic processes, including Brownian motion \cite{Davie.Fouche:2013,Fouche.Mukeru:2013,Bilokon.Edalat:2017,Collins:} and L\'evy and Feller processes \cite{Maler:2016}.

There have also been a number of papers about the computability of various theorems in measure theory, e.g.\
the ergodic theorem \cite{Avigad.Gerhardy.Towsner:2010,Hoyrup:2013},
the Riesz representation theorems \cite{Lu.Weihrauch:2008,Lu.Weihrauch:2007,Jafarikhah.Weihrauch:2013},
various decomposition theorems 
\cite{Jafarikhah.Weihrauch:2014,Hoyrup.Rojas.Weihrauch:2012},
as well as other results \cite{Pauly.Fouche:2017}.
Additional works on computable probability theory are motivated by probabilistic programming
\cite{Freer.Roy:2012,Ackerman.Freer.Roy:2017,Mislove:2016,Mislove:2018,Vakar.Kammar.Staton:2018,Ackerman.Avigad.Freer.ea:2018,Huang.Morrisett.Spitters:2018}, and others still, as we will see, are motivated by work
in algorithmic randomness.

\subsection{Algorithmic randomness\label{subsec:Algorithmic-randomness}}

Algorithmic randomness is closely tied to computable analysis, and
many researchers have focused on exploring these connections.

In the 1960s and 1970s, Solomonoff, Kolmogorov, Martin-Löf, Levin,
Schnorr, Chaitin, and others grappled with the relationship between
information theory, probability theory, dynamical systems, and computability. (See Schnorr \cite{Schnorr:1977} for a survey of that time period.)
Besides the already mentioned characterizations of Martin-Löf and
Schnorr randomness via measures and effectively open sets, there are
also characterizations of randomness via algorithmic complexity (see \cite{Downey.Hirschfeldt.Nies.ea:2006,Li.Vitanyi:2008,Nies:2009,Downey.Hirschfeldt:2010}).
This is closely connected to the work on effective Hausdorff dimension by 
Lutz, Mayordomo, and others \cite{Lutz:2000,Lutz:2003,Lutz:2005,Mayordomo:2002,Reimann:2008}.  It also led to fruitful research by V'yugin and
others connecting algorithmic complexity, entropy, dimension, and ergodic theory 
\cite{Vyugin:1998,Hochman:2009,Hoyrup:2012,Simpson:2015}.

While most work in algorithmic randomness has taken place on Cantor
space $\{0,1\}^{\mathbb{N}}$ or the unit interval with the Lebesgue
measure, there have been extensions of the theory to other spaces.
Martin-Löf \cite[\S V]{Martin-Lof:1966mz} considered Martin-Löf randomness for other
Bernoulli measures, and Schnorr \cite[Ch.~5]{Schnorr:1971rw} did the same for Schnorr randomness. 
Levin \cite{Levin73, Levin:1976uq, Levin84} 
generalized Martin-Löf randomness
to noncomputable probability measures on Cantor space.

Asarin and Prokrovskii \cite{Asarin.Pokrovskii:1986} extended Martin-Löf
randomness to Brownian motion, and this work has been taken up by
Fouché and others \cite{Fouche:2000b,Fouche:2000c,Kjos-Hanssen.Nerode:2007,Fouche:2008,Fouche:2009,Hoyrup.Rojas:2009b,Kjos-Hanssen.Nerode:2009,Kjos-Hanssen.Szabados:2011,Fouche:2014,Fouche.Mukeru.Davie:2014,Allen.Bienvenu.Slaman:2014}.
Hertling and Weihrauch \cite{Hertling.Weihrauch:2003}, Gács \cite{Gacs:2005},
and Hoyrup and Rojas \cite{Hoyrup.Rojas:2009} extended Martin-Löf's
and Levin's ideas to other computable metric spaces. Hoyrup and Rojas
\cite{Hoyrup.Rojas:2009a} also realized that the effectively measurable
functions and sets of Edalat \cite{Edalat:2009nx} could be characterized
in terms of Martin-Löf randomness. This approach is called layerwise
computability, and Hoyrup and Rojas's ideas have been extended to
Schnorr randomness by Pathak, Rojas, and Simpson \cite{Pathak.Rojas.Simpson:2014},
Miyabe \cite{Miyabe:2013uq}, and Rute \cite{Rute:2013pd}.

In Section~\ref{sec:Characterizing-randomness} we survey more results showing that
Schnorr randomness, computable randomness, and Martin-Löf randomness
can all be characterized via classical convergence theorems in analysis,
and we will highlight the powerful tools which make it easy to translate
analytic theorems into results about randomness.

\subsection{Point-free measure theory: measure algebras, locales, forcing, and
category theory\label{subsec:Point-free}}

Measure theory is usually presented in a point-set-theoretic manner:
One first develops a theory of points, sets, and functions. Then certain
sets and functions are deemed to be ``measurable''. This is, more
or less, the approach of many of the early constructivists, including
Brouwer, Demuth, Bishop, and Martin-Löf. In classical practice, one
often goes a step further, considering equivalence classes modulo almost everywhere equivalence.
For example, let $\mu$ be a measure on $\{0,1\}^{\mathbb{N}}$. Then one
has the vector space $L^{0}(\mu)$ of measurable functions modulo
$\mu$-a.e.~equivalence, the Banach space $L^{1}(\mu)$ of $\mu$-integrable
functions modulo $\mu$-a.e.~equivalence, and the complete Boolean
algebra of measurable sets modulo $\mu$-a.e.~equivalence. These
spaces are all complete separable metric spaces. 

The point-free approach to measure theory proceeds differently.  
In it, one formally defines ``measurable functions'' and
``measurable sets'' directly as objects in the above metric spaces, without
explicitly mentioning the underlying functions, sets, and points.  The ``functions''
and ``sets'' in these spaces are merely formal objects, not actual functions or sets. 

Indeed, we already saw that Šanin \cite{Sanin:1968dq} and Kosovski\u{\i}
\cite{Kosovskii:1969a,Kosovskii:1969b,Kosovskii:1969,Kosovskii:1970,Kosovskii:1973,Kosovskii:1973a} 
used this approach to reason about a large
subset of probability theory. An equivalent approach is given by Coquand
and Palmgren \cite{Coquand.Palmgren:2002}, who construct a space
of measurable sets as the metric completion of a countable Boolean
ring with a measure on it. Using this approach, they give constructive
proofs of Kolmogorov's 0-1 law, the first Borel-Cantelli lemma, and
the strong law of large numbers.
Spitters \cite{Spitters:2006} extended this approach to
include integrable and measurable functions.

This all ties in to point-free topology, a field which has close connections
to constructive mathematics (see Section~5 of \cite{sep-mathematics-constructive}). 
One type of point-free space, generalizing topological spaces,
is a \emph{locale}. A locale is given by a partial order which behaves
like the partial order of open sets in a topological space under the
subset relationship \textemdash{} this partial order has top and bottom
elements, is closed under arbitrary joins $\bigcup$ and finite meets
$\cap$, and satisfies the distributive law $U\cap\left(\bigcup_{i\in I}V_{i}\right)=\bigcup_{i\in I}\left(U\cap V_{i}\right)$.
A \emph{morphism} $f\colon X\rightarrow Y$ between locales $X$
and $Y$ behaves like a continuous function between topological spaces;
formally it is given by a map from the ``open sets'' of $Y$ to
the ``open sets'' of $X$ which preserves finite meets, and arbitrary
joins. If $\mu$ is a Borel probability measure on $[0,1]$, the measurable
sets modulo $\mu$-a.e.~equivalence form a locale, the \emph{$\mu$-measurable locale}.\footnote{%
Recall that the Boolean algebra of measure sets modulo a.e.~equivalence
is complete, and therefore closed under arbitrary joins, not just countable joins.}
If we denote the $\mu$-measurable
locale as $(\{0,1\}^{\mathbb{N}},\mu)$, then the morphisms $f\colon(\{0,1\}^{\mathbb{N}},\mu)\rightarrow\mathbb{R}$
(where $\mathbb{R}$ has the standard topology/locale) are exactly
the measurable functions modulo $\mu$-a.e.~equivalence.\footnote{%
While we are not aware of a fully constructive treatment of the $\mu$-measurable locale, 
we note that none of the constructive definitions of measurable set given so far are 
constructively closed under infinite countable unions.  
Nonetheless, we suggest as a candidate the locale whose ``open sets'' are given 
by the representation $\delta_+$ in \cite{Weihrauch.Tavana:2014,Weihrauch:2017} 
of point-free measurable sets computable from below.  
Computably, this has the closure properties of a $\sigma$-locale (\cite[Thm.~4.1]{Weihrauch:2017}) 
and the computable morphisms $f\colon(\{0,1\}^{\mathbb{N}},\mu)\rightarrow\mathbb{R}$ 
are exactly the point-free measurable functions of Šanin and others 
(see the representation $\delta_\textrm{mfo}$ in \cite{Weihrauch:2017}).}
(Notice, 
that if $\mu$ is the Lebesgue measure, the $\mu$-measurable locale
is not homeomorphic to any topological space\footnote{Assume the locale $(\{0,1\}^{\mathbb{N}},\mu)$ is homeomorphic to a topological
space $X$. For each measurable set $B$ of $(\{0,1\}^{\mathbb{N}},\mu)$,
let $\widehat{B}$ be the corresponding open set in $X$. Consider
a point $x\in X$. For each $k$, there is exactly one $\sigma\in\{0,1\}^{k}$
such that $x\in\widehat{[\sigma]}$. Existence follows from $\bigcup\{\widehat{[\sigma]} : \sigma\in\{0,1\}^{k}\}=\widehat{\{0,1\}^{\mathbb{N}}}=X$.
Uniqueness follows from $\widehat{[\sigma]}\cap\widehat{[\tau]}=\widehat{\varnothing}=\varnothing$.
Let $U_{k}=\bigcup \{[\sigma] : \sigma\in\{0,1\}^{k},x\notin\widehat{[\sigma]}\}$.
Then $\mu(U_{k})=1-2^{-k}$. Since, $\mu(\bigcup_{k}U_{k})=1$, we
have $x\in\bigcup_{k}\widehat{U_{k}}$ contradicting the definition
of $U_{k}$.}, necessitating the use of point-free methods.)

Not only can one reason about measure theory in the locale of $\mu$-measurable
sets, but one can also use the measurable locale to give a rigorous
formulation of randomness. One can naively view probability theory
as the study of random events, whereby a random event is one satisfying
every probability one property. While such ``random events'' do
not actually exist, the measurable locale can be viewed as the space
of random points.

This ties in closely with set-theoretic forcing. In forcing one has
two mathematical universes $\mathcal{U}\subseteq\mathcal{V}$, the
smaller of which is known as the \emph{ground model}. If one takes
a locale $L$ in the ground model, forcing allows one to construct
objects $g$ in the larger universe, called \emph{generics}, which
behave as if they are ``points'' in the ``space'' $L$. In Solovay forcing \cite[Ch.~26]{Jech:2003},
 one forces with the $\mu$-measurable locale (also known as the measure
algebra of $\mu$-measurable sets). The resulting generics are known
as \emph{Solovay randoms}. Being a Solovay random is equivalent to
being in every $\mu$-measure one set in the ground model. We now
have an analogy to Schnorr randomness, which is equivalent to being
in every constructive $\mu$-measure one set. In Subsection~\ref{subsec:Forcing}
we strengthen this analogy by giving an effective version of Solovay
forcing, where the generics are the Schnorr randoms.

Simpson \cite{Simpson:2012} has proposed another locale as a model
for randomness. The \emph{locale of random sequences} is the locale
of open sets of $\{0,1\}^{\mathbb{N}}$ modulo a.e.~equivalence. 
This locale is analogous to Kurtz randomness. (Recall, a point is Kurtz
random if it is in every measure one effectively open set.) Like Kurtz
randomness, the locale of random sequences does not always satisfy
the strong law of large numbers \cite{Simpson:2009b}.
This analogy can also be made formal with forcing.

Locales and forcing are part of a larger categorical framework, including
sheaves, toposes, type theory, and other tools important to modern constructive
mathematics. There is new work approaching measure theory and 
probability from this perspective 
\cite{Jackson:2006,Rodrigues:2009,Vickers:2011,Simpson:2017,Simpson:,Faissole.Spitters:, Clark:math-overflow,nlab:probability-theory}, 
much of it building on the work of Giry \cite{Giry:1982}. While this work
is in progress, we conjecture that in these settings, questions about
randomness will once again naturally arise, both implicitly and explicitly.
To the extent that these categorical models are reasoned about constructively
or computably, we will again find connections and analogies with algorithmic
randomness.

\section{\label{sec:Characterizing-randomness}Characterizing algorithmic
randomness via theorems in analysis}
One of the most important characteristics of algorithmic randomness
is that it satisfies many of the almost everywhere theorems of mathematics.
For example, every Schnorr random (and therefore every Martin-Löf
random) satisfies the strong law of large numbers \textemdash{} that
is the sequence of binary digits $(x_{n})$ of $x\in\{0,1\}^\mathbb{N}$ satisfies
$\lim_n\frac{1}{n}\sum_{k=0}^{n-1}x_k=\frac{1}{2}$. However, the strong
law of large numbers, or even the more advanced law of the iterated
logarithm, does not characterize Schnorr randomness. This is simply
because one can construct a computable sequence $x\in\{0,1\}^\mathbb{N}$ for which
both theorems hold \cite{Pincus.Singer:2012}.

However, it turns out that many of the more general theorems in analysis
and probability, usually involving a free parameter, do characterize the
standard algorithmic randomness notions. These characterization results show that
Martin-Löf randomness, computable randomness, and Schnorr randomness
are all natural randomness notions. What follows is a survey of some
of these results.

\subsection{Monotone convergence\label{subsec:Monotone-convergence}}

A variation of the monotone convergence theorem in measure theory
states that given an increasing sequence of continuous nonnegative
functions $g_{n}\colon[0,1]\rightarrow[0,\infty)$, if $\sup_{n}\int_{0}^{1}g_{n}\,dx$
is finite, then $\sup_{n}g_{n}(x)<\infty$ for almost every $x$.
This can be used to characterize Schnorr randomness and Martin-Löf
randomness.
\begin{thm}[Levin \cite{Levin:1976uq}]
\label{thm:MLR-monotone}The following are equivalent for a real
$x\in[0,1]$.

\begin{enumerate}
\item The real $x$ is Martin-Löf random.
\item The supremum $\sup_{n}g_{n}(x)$ is finite for every increasing computable
sequence of continuous functions $g_{n}\colon[0,1]\rightarrow[0,\infty)$
such that $\sup_{n}\int_{0}^{1}g_{n}(x)\,dx$ is finite.
\end{enumerate}
Moreover, a set $E$ is a Martin-Löf null set if and only if $E\subseteq\{x:\lim_n g_n(x)=\infty\}$
for some such sequence $(g_{n})$.
\end{thm}

\begin{thm}[Rute \cite{Rute:}]
\label{thm:CR-monotone}The following are equivalent for a real $x\in[0,1]$.

\begin{enumerate}
\item The real $x$ is Schnorr random.
\item The supremum $\sup_{n}g_{n}(x)$ is finite for every increasing computable
sequence of continuous functions $g_{n}\colon[0,1]\rightarrow[0,\infty)$
such that there is some computable probability measure $\mu$ 
such that $\int_{A}g_{n}(x)\,dx\leq\mu(A)$ for all Borel sets $A\subseteq[0,1]$.
\end{enumerate}
Moreover, a set $E$ is a Schnorr null set if and only if $E\subseteq\{x:\lim_{n}g_n(x)=\infty\}$
for some such sequence $(g_{n})$.
	
\end{thm}

\begin{thm}[Miyabe \cite{Miyabe:2013uq}]
\label{thm:SR-monotone}The following are equivalent for a real $x\in[0,1]$.

\begin{enumerate}
\item The real $x$ is Schnorr random.
\item The supremum $\sup_{n}g_{n}(x)$ is finite for every increasing computable
sequence of continuous functions $g_{n}\colon[0,1]\rightarrow[0,\infty)$
such that $\sup_{n}\int_{0}^{1}g_{n}(x)\,dx$ is finite and computable.
\end{enumerate}
Moreover, a set $E$ is a Schnorr null set if and only if $E\subseteq\{x:\lim_{n}g_n(x)=\infty\}$
for some such sequence $(g_{n})$.

\end{thm}

\subsection{Differentiability\label{subsec:Differentiability}}

A theorem of Lebesgue states that every function of bounded variation
is differentiable almost everywhere. A function $f\colon[a,b]\to\mathbb{R}$ is of \emph{bounded variation}
if there is a bound $c$ such that for all $a\leq x_{0}<\ldots<x_{n}\leq b$,
one has $\sum_{i=0}^{n-1}|f(x_{i})-f(x_{i+1})|\leq c$. The minimum
such bound is the variation $\Var_a^b(f)$.
\begin{thm}[{($\Rightarrow$) Demuth \cite{Demuth:1975zr}, ($\Leftarrow$) Brattka,
Miller, Nies \cite{BrattkaMillerNies16}}]
\label{thm:MLR-bdd-var}The following are equivalent for a real $x\in[0,1]$.

\begin{enumerate}
\item The real $x$ is Martin-Löf random.
\item The function $f$ is differentiable at $x$ for every computable function
$f\colon[0,1]\rightarrow\mathbb{R}$ of bounded variation.
\end{enumerate}
\end{thm}

\begin{thm}
\label{thm:CR-bdd-var}The following are equivalent for a real $x\in[0,1]$.

\begin{enumerate}
\item The real $x$ is computably random.
\item The function $f$ is differentiable at $x$ for every computable function
$f\colon[0,1]\rightarrow\mathbb{R}$ of bounded variation with a computable
variation $\Var_0^1(f)$.
\end{enumerate}
\begin{proof}
Brattka, Miller, and Nies \cite[Cor.~4.3]{BrattkaMillerNies16} proved this theorem for
nondecreasing $f$.  Therefore it is sufficient to find two nondecreasing computable functions $f^+$ and $f^-$ such that $f = f^+ - f^-$.   Let $f^+ = \Var_0^x(f)$ and $f^- = \Var_0^x(f) - f$.  Both are non-decreasing.  Since $\Var_0^1(f)$ is computable, so is $\Var_0^x(f)$.  (Indeed, $\Var_0^x(f)$ is both computable from below, and computable from above by the calculation $\Var_0^x(f) = \Var_0^1(f) - \Var_x^1(f).$)
\end{proof}

\end{thm}
In this next result, a function $f \colon [0,1] \to \mathbb{R}$ 
is \emph{effectively integrable} if there
is a computable sequence of rational polynomials $p_{n}$ such that
$\|f-p_{n}\|_{L^{1}}=\int(f-p_{n})\,d\mu\leq2^{-n}$.\footnote{
In Section~\ref{sec:Foundations} we provide three different 
definitions of ``effectively integrable function''. This is the point-free version.  
Many authors refer to these as \emph{$L^1$-computable functions}.}
\begin{thm}[Rute {{\cite[Cor.~4.17, p.~48, Cor.~12.5, p.~67]{Rute:2013pd}}}]
\label{thm:SR-bdd-var}The following are equivalent for a real $x\in[0,1]$.

\begin{enumerate}
\item The real $x$ is Schnorr random.
\item The function $f$ is differentiable at $x$ for every computable function
$f\colon[0,1]\rightarrow\mathbb{R}$ of bounded variation with
effectively integrable derivative $f'$.
\end{enumerate}
\end{thm}
Now, let us consider Rademacher's theorem that says that every Lipchitz
function is almost everywhere differentiable. Recall, a function
$f\colon[0,1]\rightarrow\mathbb{R}$ is \emph{Lipschitz} if there is a constant
$C > 0$ such that for all $x,y\in[0,1]$, $|f(x)-f(y)|\leq C|x-y|$.
\begin{thm}[Freer, Kjos-Hannsen, Nies, Stephan \cite{Freer.Kjos-Hanssen.Nies.ea:2014}]
\label{thm:CR-lipschitz}The following are equivalent for a real
$x\in[0,1]$.

\begin{enumerate}
\item The real $x$ is computably random.
\item Every computable Lipschitz function $f\colon[0,1]\rightarrow\mathbb{R}$
is differentiable at $x$.
\end{enumerate}
\end{thm}

Lebesgue's differentiation theorem states that if $f\colon [0,1] \to \mathbb{R}$ is integrable,
then $\frac{1}{2r}\int_{x-r}^{x+r}f(y)\,dy$ converges to
$f(x)$ as $r\rightarrow0$ for almost every $x$.
\begin{thm}[Pathak, Rojas, Simpson \cite{Pathak.Rojas.Simpson:2014}, Rute \cite{Rute:2013pd}]
\label{thm:SR-LDT}The following are equivalent for a real $x\in[0,1]$.

\begin{enumerate}
\item The real $x$ is Schnorr random.
\item The averages $\frac{1}{2r}\int_{x-r}^{x+r}f(y)\,dy$ converge
as $r\rightarrow0$ for every effectively integrable function $f$.
\end{enumerate}
\end{thm}

This version of the Lebesgue's differentiation theorem also holds in multiple dimensions.  
In Theorem~\ref{thm:SR-LDT-2} we will address the question,
``To which value does $\frac{1}{2r}\int_{x-r}^{x+r}f(y)\,dy$ converge?''

\subsection{Martingale theory\label{subsec:Martingale-theory}}

In this subsection, we will work in the fair-coin measure on
Cantor space $\{0,1\}^\mathbb{N}$ for convenience.
If $f$ is an integrable function and $\mathcal{F}$ is a $\sigma$-algebra,
then the conditional expectation $\mathbb{E}[f\mid\mathcal{F}]$ is
the unique (up to a.e.\ equivalence) integrable function $g$ such
that $\int_{A}g\,d\mu=\int_{A}f\,d\mu$ for all $A\in\mathcal{F}$.
If $\mathcal{F}$ is the least $\sigma$-algebra for which the functions
$h_{0},\ldots,h_{n-1}$ are $\mathcal{F}$-measurable, then we write
$\mathbb{E}[f\mid h_{0},\ldots,h_{n-1}]=\mathbb{E}[f\mid\mathcal{F}]$.
A martingale is a sequence of integrable functions $(f_{n})_{n\in\mathbb{N}}$
such that for all $n \ge 1$, 
\[
\mathbb{E}[f_{n}\mid f_{0},\ldots,f_{n-1}]=f_{n-1}\quad\mu\text{-a.e.}
\]

Doob's martingale convergence theorem states that if $(f_{n})$ is
a martingale such that $\sup_{n}\|f_{n}\|_{L^{1}}<\infty$, then $f_{n}(x)$
converges for almost every $x$. We will say that a martingale $(f_{n})$
is computable if $(f_{n})$ is a computable sequence of computable
functions.  
(The next two theorems can be strengthened to include martingales on 
arbitrary computable probability measures as in Subsection~\ref{subsec:Computable-measures}
where the functions $f_n$ are Brouwer/Schnorr effectively measurable as in 
Subsection~\ref{subsec:Pointwise-approaches}. 
See footnote~10 (p.~33) and Theorem~7.11 (p.~55) in Rute \cite{Rute:2013pd}.)
\begin{thm}[Takahashi \cite{Takahashi:2005}, Merkle, Mihalovi\'{c}, Slaman \cite{Merkle.Mihailovic.Slaman:2006}]
\label{thm:MLR-mart-conv}The following are equivalent for a sequence
$x\in\{0,1\}^\mathbb{N}$.

\begin{enumerate}
\item The sequence $x$ is Martin-Löf random.
\item The sequence $f_{n}(x)$ converges for every computable martingale
$(f_{n})$ such that $\sup_{n}\|f_{n}\|_{L^{1}}$ is finite.
\end{enumerate}
\end{thm}

\begin{thm}[Rute {{\cite[Thm.~7.11, p.~55, Thm.~12.9, p.~68]{Rute:2013pd}}}]
\label{thm:SR-mart-conv}The following are equivalent for a sequence
$x\in\{0,1\}^\mathbb{N}$.

\begin{enumerate}
\item The sequence $x$ is Schnorr random.
\item The sequence $f_{n}(x)$ converges for every computable martingale
$(f_{n})$ such that $\sup_{n}\|f_{n}\|_{L^{1}}$ is finite and computable
and such that $\lim_{n}f_{n}$ is effectively integrable.
\end{enumerate}
\end{thm}

Most of the martingale work in algorithmic randomness, however, has been focused
on \emph{computable dyadic martingales} (often just called \emph{computable martingales}), 
that is, martingales $f_n$ of the form
$f_n(x) = g(x \upharpoonright n)$ for some computable function $g$.  This provides
another convenient characterization of computable randomness.

\begin{thm}[Folklore {{\cite[Theorem~7.1.3]{Downey.Hirschfeldt:2010}}}, following Schnorr \cite{Schnorr1971a}]
\label{thm:CR-mart-conv}The following are equivalent for a sequence
$x\in\{0,1\}^\mathbb{N}$.

\begin{enumerate}
\item The sequence $x$ is computably random.
\item The sequence $f_{n}(x)$ converges for every nonnegative computable dyadic martingale
$(f_{n})$.
\end{enumerate}
\end{thm}

\subsection{Ergodic theory\label{subsec:Ergodic-theory}}

Again, we work in the fair-coin measure on Cantor space.
A measure preserving transformation $T\colon\{0,1\}^\mathbb{N}\rightarrow\{0,1\}^\mathbb{N}$ is
a measurable map such that $\mu(T^{-1}(A))=\mu(A)$ for all measurable
sets $A$. The pointwise ergodic theorem states that for any integrable
function $f\colon\{0,1\}^\mathbb{N}\rightarrow\mathbb{R}$, the following average converges
for almost every $x$. 
\[
\frac{1}{n}\sum_{k=0}^{n-1}f(T^{k}x).
\]

In this next theorem, an \emph{almost everywhere computable map}
is one which is computable on a $\Pi_{2}^{0}$ set of measure one.
(Every result in this subsection concerning almost everywhere computable
maps also holds for the more general Brouwer/Schnorr effectively
measurable maps that we describe in Subsection~\ref{subsec:Pointwise-approaches}.
The results also extend to computable probability measures 
on computable metric spaces as discussed in Subsection~\ref{subsec:Computable-measures}.  
For full generalizations of the next two theorems, see Hoyrup and Rojas \cite[Thm.~8]{Hoyrup.Rojas:2009a}
and Rute \cite[p.~72]{Rute:2013pd}, respectively.)

\begin{thm}[V'yugin \cite{Vyugin:1998}, Franklin, Towsner \cite{Franklin.Towsner:2014}]
\label{thm:MLR-ergodic}The following are equivalent for a sequence $x\in\{0,1\}^\mathbb{N}$.

\begin{enumerate}
\item The sequence $x$ is Martin-Löf random.
\item The ergodic averages 
\[
\frac{1}{n}\sum_{k=0}^{n-1}f(T^{k}(x))
\]
converge for every integrable, a.e.~computable $f\colon\{0,1\}^\mathbb{N}\rightarrow\mathbb{R}$
and for every a.e.~computable measure-preserving $T$.
\end{enumerate}
\end{thm}
A measure preserving transformation is \emph{ergodic} if and only
if $T^{-1}(A)=A$ implies that $\mu(A)$ is $0$ or $1$. For ergodic
$T$, the ergodic theorem states that almost surely 
\[
\lim_{n\to\infty}\frac{1}{n}\sum_{k=0}^{n-1}f(T^{k}(x))=\int f\,d\mu.
\]

\begin{thm}[Gács, Hoyrup, Rojas \cite{Gacs.Hoyrup.Rojas:2011}]
\label{thm:SR-ergodic}The following are equivalent for a sequence $x\in[0,1]$.

\begin{enumerate}
\item The sequence $x$ is Schnorr random.
\item The ergodic averages 
\[
\frac{1}{n}\sum_{k=0}^{n-1}f(T^{k}(x))
\]
converge for every a.e.~computable $f\colon\{0,1\}^\mathbb{N}\rightarrow\mathbb{R}$
which is effectively integrable and for every a.e.~computable
ergodic measure-preserving $T$.
\end{enumerate}
\end{thm}
A special case of the ergodic theorem is the strong law of large numbers (SLLN).
As mentioned above, SLLN alone does not characterize any algorithmic randomness notions. 
Nonetheless, starting with Von Mises \cite{von-Mises:1919}, there have been attempts to 
define randomness by requiring that a sequence $x$ not only satisfy SLLN, 
but that certain transformations $T(x)$ of that sequence do as well.
The transformations that Von Mises considered were subsequences of $x$ 
given by a \emph{selection rule}; that is, one has to choose whether to select 
the bit $x_i$ based only on the values of the former bits $x_0, ... x_{i-1}$.
A \emph{Church stochastic sequence} is the formulation of this notion
where the selection rules are computable, that is given by
a computable function $f\colon \{0,1\}^{<\mathbb{N}} \to \{\text{yes}, \text{no}\}$
\cite[Def.~7.4.1]{Downey.Hirschfeldt:2010}.
While this stochasticity notion and its generalizations are not as 
useful as the established notions of randomness, Schnorr realized that 
Von Mises's ideas can be used to define Schnorr randomness 
if one uses the correct class of transformations $T(x)$.
\begin{thm}[{Schnorr \cite[Thm.~12.1]{Schnorr:1971rw}}]\label{thm:sch-stoch}
\label{thm:SR-map-SLLN}The following are equivalent for a sequence $x\in\{0,1\}^\mathbb{N}$.

\begin{enumerate}
\item The sequence $x$ is Schnorr random.
\item The frequency of $1$s in $T(x)$ converges to $1/2$, i.e.
\[
\lim_{n\to\infty}\frac{1}{n}\sum_{k=0}^{n-1}(T(x))_{k}=\frac{1}{2},
\]
for every a.e.~computable measure-preserving map $T\colon[0,1]\rightarrow[0,1]$
where $(T(x))_{k}$ is the $k$th bit of $T(x)$.
\end{enumerate}
\end{thm}
Note, not every Church selection rule corresponds to an a.e.~computable 
measure-preserving map.  While, the Church selection rules are total functions
$f\colon \{0,1\}^{<\mathbb{N}} \to \{\text{yes}, \text{no}\}$, the corresponding 
transformation $T\colon \{0,1\}^\mathbb{N} \to \{0,1\}^\mathbb{N}$ may be partial, 
and the measure of the domain of $T$ may be less than one.
Indeed, Schnorr randomness and Church stochasticity are incomparable 
notions \cite[\S8.4]{Downey.Hirschfeldt:2010}.

\subsection{Some additional remarks\label{subsec:Additional-remarks}}

The above results show that each of Schnorr randomness, computable
randomness, and Martin-Löf randomness can be characterized naturally
via theorems from analysis. However, if one looks at the proofs,
for the most part these results can be rewritten in terms of effective
null sets. For example, Theorem~\ref{thm:MLR-bdd-var} 
can be adapted as follows.
\begin{thm}
\label{thm:ML-null-set-bdd-var}For each computable function $f\colon[0,1]\rightarrow\mathbb{R}$
of bounded variation, the set $\{x\in[0,1]:f\ \text{is not differentiable at}\ x\}$
is a Martin-Löf null set. Conversely, for each Martin-Löf null set
$A$, there is a computable function $f\colon[0,1]\rightarrow\mathbb{R}$
of bounded variation such that $A\subseteq\{x\in[0,1]:f\ \text{is not differentiable at}\ x\}$.
\end{thm}
By relativizing the second part of this theorem, one gets the following
corollary.
\begin{cor}
\label{cor:null-set-bbd-var}For each null set $A\subseteq [0,1]$, there is a continuous function $f\colon[0,1]\rightarrow\mathbb{R}$
of bounded variation such that $A\subseteq\{x\in[0,1]:f\ \text{is not differentiable at}\ x\}$.
\end{cor}

What this logic tells us is that if we have a result which characterizes all Martin-Löf randoms (e.g.Theorem~\ref{thm:MLR-bdd-var}), it should relativize to a result 
(e.g.\ Corollary~\ref{cor:null-set-bbd-var}) which characterizes \emph{all null sets}.  The same
holds for Schnorr and computable randomness, or any other notion which deserves to be called a 
``randomness notion''.  

This allows us to instantly rule out some theorems as those which characterize 
randomness notions.  For example, the strong law of large numbers only characterizes a single null
set, namely the set of numbers which are not simply normal in base 2.  Therefore, there is no algorithmic 
randomness notion characterized by the strong law of large numbers.

A more interesting example is a theorem of Weyl.  Given a sequence of distinct integers $(a_n)$,
the set $\{a_n x\}_n$ is uniformly distributed modulo one for almost every $x\in [0,1]$ .  
Avigad \cite{Avigad:2013kx} defined a real $x \in [0,1]$ to be \emph{UD-random} if 
$\{a_n x\}_n$ is uniformly distributed modulo one for all computable sequences $(a_n)$ of distinct integers.  
However, Avigad noticed that there is a specific null set $C$ such that for every 
sequence $(a_n)$ of distinct integers (not necessarily computable), there is a real $x \in C$ where
$\{a_n x\}_n$ is uniformly distributed modulo one.  Hence it is impossible to use Weyl's theorem to
characterize null sets, and ``UD randomness'' is not a true notion of randomness (for the Lebesgue measure).

So far we have been talking about randomness relative to the Lebesgue measure.
It is possible that Weyl's theorem characterizes null sets for a different measure $\mu$ on $[0,1]$.
It is also possible that UD randomness is not associated with null sets for a single measure, 
but instead sets which are null for all measures in a family of measures (see, for example, 
\cite{Reimann:2008,Bienvenu.Gacs.Hoyrup.ea:2011}).  Indeed, every type of ``exceptional set'' in mathematics has its
own notion of effectively random-like objects.  For example, effective Cohen genericity corresponds to meager sets.  Kurtz randomness (which does not behave like a typical randomness notion) corresponds to subsets of null $F_\sigma$ sets (that is, a countable union of closed sets).  Such sets are both null and meager.

\subsection{Connections with constructive and reverse mathematics\label{subsec:Connections-with-constructive}}

These characterization theorems have a close connection to constructive
mathematics and reverse mathematics. (Some even call this approach
``reverse randomness'' because of the similarities.) For example,
consider the Lebesgue differentiation theorem.  It is constructive,
as shown by Bishop \cite[Ch.~8, Thm.~5]{Bishop:1967lq} 
and Demuth \cite[Thm.~4.4]{Demuth.Kucera:1979},
and it also holds of Schnorr randomness (Theorem~\ref{thm:SR-LDT}).
Conversely, the nonconstructive theorems such as the ergodic theorem
and the martingale convergence theorem do not hold for all Schnorr randoms
(Theorems~\ref{thm:MLR-ergodic} and \ref{thm:MLR-mart-conv}).
This is not a coincidence, but instead a fundamental connection
between Schnorr randomness and constructive mathematics.
\begin{namedthm}[Informal Principle 1]
Consider an a.e.~theorem $T$ of the form 
\begin{quote}
for all objects $a$, for almost every $x$, it holds that $P(x,a)$ 
\end{quote}
\noindent where ``almost every $x$'' is defined using the constructive
null sets of Brouwer, Demuth, or Bishop (recall that, under a computable interpretation, 
these are basically Schnorr null sets). If $T$ is constructively
provable, then $P(x,a)$
holds for all Schnorr randoms $x$ and all computable objects $a$.
\end{namedthm}
\begin{proof}[Informal justification]
Assume $T$ is constructively provable. Fix a computable $a$. From a constructive
proof of $T$ one can explicitly construct a Schnorr null set $N$,
for which if $P(x,a)$ does not hold then $x \in N$.  
Therefore, $P(x,a)$ holds for all Schnorr randoms $x$.
\end{proof}

A common special case of the above principle is a.e.\ convergence, which
by Ergoroff's theorem is classically equivalent to almost uniform convergence
(when working in a probability space).  Most constructive proofs of a.e.\ convergence
proceed through almost uniform convergence.
\begin{namedthm}[Informal Principle 2]
Consider an almost uniform convergence theorem $T$ of the form 
\begin{quote}
given a sequence $(f_n)$ of uniformly continuous functions 
$f_n \colon [0,1] \to \mathbb{R}$, satisfying some property $P((f_n))$, 
then the sequence $(f_n)$ converges almost uniformly
\end{quote}
\noindent where ``almost uniformly'' is defined using the constructive
definitions of Brouwer, Demuth, Bishop, or Šanin 
(see Definitions~\ref{def:pointfree-au-conv} and \ref{def:MLR-SR-eff-au-conv}).
If $T$ is constructively provable, then $f_n(x)$
converges for all Schnorr randoms $x$ and all computable sequences of computable
functions $f_n$ such that $P((f_n))$ holds and is effectively realized.
\end{namedthm}
\begin{proof}[Informal justification]
Bishop constructively observed that if $(f_n)$ converges almost uniformly then $f_n(x)$ convergences 
for almost every $x$ \cite[p.~196]{Bishop:1967lq}.  The rest follows from Informal Principle 1.

An alternate justification is as follows.  From the constructive proof of $T$ and a realizer of $P((f_n))$ we can extract
a computable rate of almost uniform convergence.  From this, we can apply 
the result that an effective rate of almost uniform convergence is sufficient to
show convergence on Schnorr randoms. (This is due to Hoyrup, Rojas, Galatolo \cite[Theorem 1]{Galatolo.Hoyrup.Rojas:2010a} and  Rute \cite[Lemma~3.19, p.~41]{Rute:2013pd}.  
Also, see Lemma~\ref{lem:useful-conversions}(\ref{item:ua-implies-SR}).)
\end{proof}
\begin{rem}
We would like to regard these previous two results as informal recipes
for translating a constructive result into one about Schnorr randomness,
rather than true meta-theorems. The constructive systems of Bishop and
others are not given by formal axioms, making it difficult to truly
formalize this result. Also, there are small subtleties, such as what
it means for $a$ to be computable or $P((f_n))$ to be effectively realized, 
that are not worth considering here.

Also while it is convenient that the definition of null set used in, say,
Bishop's work is equivalent to that used by Schnorr, it is not strictly
necessary for the above results to hold.  Even Martin-Löf's constructive
theorems or the point-free theorems of Šanin can be used to extract
computable results about Schnorr randomness.  
See Section~\ref{sec:Foundations} for details on how to effectively
convert between the different definitions. 
\end{rem}
The converse of our informal principle is technically not true. For
example, consider the theorem that for every monotone sequence of
bounded continuous functions $g_{n}\colon[0,1]\rightarrow[0,1]$, the sequence
$g_{n}(x)$ converges for almost every $x$. This theorem is not constructive, but it is true that for every
computable sequence of computable functions $(g_{n})$, the sequence $g_{n}(x)$
converges for every Schnorr random $x$ (and indeed every $x\in[0,1]$).

Nonetheless, the converse of our informal principle seems to be ``true is spirit''.
The natural a.e.~theorems holding for Schnorr randomness \textemdash{}
the law of large numbers, the Lebesgue differentiation theorem, the
ergodic theorem for ergodic measures, etc.\ \textemdash{} are provable
in constructive mathematics.

Moreover, reverse-mathematics-type results seem to shed light on the
connections between Schnorr randomness and a.e.\ convergence theorems.
For example, the following are constructively equivalent for an increasing
sequence of nonnegative uniformly continuous functions $g_{n}\colon[0,1]\rightarrow[0,\infty)$
such that $\int_{0}^{1}g_{n}(x)\,dx$ is bounded (Bishop \cite[Ch.~7, Thm.~5]{Bishop:1967lq}). 
\begin{enumerate}
\item $\int_{0}^{1}g_{n}(x)\,dx$ converges.
\item $g_{n}$ converges almost uniformly.
\end{enumerate}
What is interesting about this result is the following connections
to Theorem~\ref{thm:SR-monotone}.
\begin{enumerate}
\item Using the second Informal Principle, one can use the forward 
direction of Bishop's result to get the forward direction of 
Theorem~\ref{thm:SR-monotone}.
\item Bishop's result suggests (but does not alone prove!)\ that one cannot remove the condition
that $\sup_{n}\int_{0}^{1}g_{n}(x)\,dx$ is computable from Theorem~\ref{thm:SR-monotone}. 
(We know this is true by Theorem~\ref{thm:MLR-monotone} along with the fact that
Schnorr randomness and Martin-Löf randomness are different.)
\item Bishop's result suggests (but does not prove!)\ that there is no other stronger
``reasonable hypotheses'' one can place on $\int_{0}^{1}g_{n}(x)\,dx$
in Theorem~\ref{thm:SR-monotone}. (This is too vague to be provable,
but it agrees with experience. While there are examples, as above,
where $g_{n}(x)$ converges on all Schnorr randoms and $\int_{0}^{1}g_{n}(x)\,dx$
is not computable, these seem contrived and unnatural.)
\end{enumerate}
It would be interesting to explore this connection more. For another
example, Spitters \cite[Thm.~16]{Spitters:2006a}
gave a constructive characterization for when ergodic averages converge.
This characterization aligns well with experience about Schnorr randomness
and the ergodic decomposition (Rute \cite[Thm.~10.2, p.~72]{Rute:2013pd},
also see Hoyrup \cite{Hoyrup:2013}).

As for the theorems which characterize Martin-Löf randomness \textemdash{}
Lebesgue's theorem for functions of bounded variation, the martingale
convergence theorem, and the ergodic theorem \textemdash{} these results
are all nonconstructive. (See Problems~3, 9, and 11 in Bishop~\cite[pp.~242--243]{Bishop:1967lq}
as well as various computability theoretic counterexamples \cite{Vyugin:2001a,BrattkaMillerNies16,Avigad.Gerhardy.Towsner:2010}.) 
Nonetheless, Bishop \cite[Ch.~8, \S3]{Bishop:1967lq}
showed that these three theorems can be made constructive by weakening
the conclusion but not the hypothesis. For these ``equal hypothesis
results'', Bishop used upcrossings. 

Pick two rationals $a<b$. A sequence of real numbers $(x_{n})$
has at least $k$ \emph{$(a,b)$-upcrossings} if there are indices
$m_{1}<n_{1}<\cdots<m_{k}<n_{k}$ such that $x_{m_{j}}<a<b<x_{n_{j}}$ 
for all $j \in \{1,\ldots, k\}$.
Classically, a bounded sequence converges exactly if for each pair
of rationals $a<b$, the number of $(a,b)$-upcrossings is bounded.
Therefore, a sequence of measurable functions $(f_{n})$ converges
almost everywhere if there is an upper bound on both $\int|f_{n}|\,d\mu$
and $\int U_{a,b}\,d\mu$ where $U_{a,b}(x)$ is the number of $(a,b)$-upcrossings
of $(f_{n}(x))$.

Doob's \emph{nonconstructive} proof of martingale convergence proceeded
via a \emph{constructive} proof of an upcrossing inequality bounding
$\int U_{a,b}\,d\mu$ \cite{Bishop:1966}. Bishop \cite[\S 8.3]{Bishop:1967lq}\cite{Bishop:1966,Bishop:1967a}, in turn,
gave constructive upcrossing inequalities for both the ergodic theorem
and Lebesgue's theorem concerning the differentiability of bounded
variation functions. Indeed, V'yugin \cite{Vyugin:1998} used the former
to prove that the ergodic theorem holds for Martin-Löf randoms (Theorem~\ref{thm:MLR-ergodic}).
Similarly, the latter can be used to give an alternate proof of Demuth's
result (Theorem~\ref{thm:MLR-bdd-var}) that Lebesgue's theorem holds
for Martin-Löf randomness.


As for reverse mathematics, given the close connection between $\mathsf{WWKL}$
and Martin-Löf randomness, one may expect that theorems such as
the pointwise ergodic theorem are equivalent to $\mathsf{WWKL}$ over
$\mathsf{RCA}_{0}$. However, this depends on how one formalizes the theorem.  
In $\mathsf{RCA}_{0}$, there are two nonequivalent ways
to say that a sequence $(x_n)$ converges.  One way is to say that $\lim_n x_n$
exists.  Using this limit characterization of convergence,
Simic \cite{Simic:2007} showed that the pointwise ergodic theorem
is equivalent to $\mathsf{ACA}$ over $\mathsf{RCA}_{0}$.  (The main idea
is that there is a computable ergodic system whose limit is Turing equivalent to $\emptyset'$.)
The other characterization of convergence is to say that $(x_n)$ is Cauchy.
Using this Cauchy characterization of convergence (in the definition of differentiable),
Nies, Triplett, and Yokoyama \cite{Nies.Triplett.Yokoyama:2017}
showed that Lebesgue's theorem about the differentiability of functions of bounded
variation is equivalent to $\mathsf{WWKL}$ over $\mathsf{RCA}_{0}$.
It is natural to conjecture that for each of the pointwise ergodic theorem, 
the martingale convergence theorem, and Lebesgue's theorem about the differentiability of functions of bounded variation, that the ``limit'' version is equivalent to $\mathsf{ACA}$
and the Cauchy version is equivalent to $\mathsf{WWKL}$ over $\mathsf{RCA}_{0}$.


Last, another way that constructive mathematics sheds light on algorithmic
randomness is via relativization. A real $x$ is \emph{Martin-Löf random relative to}
a real $y$, if $x$ is not contained in any Martin-Löf null set \emph{computable from}
$y$. While, at first, this may seem natural, it does not necessarily
agree with constructive mathematics. In constructive mathematics, when one
says that an object $A$ exists given another object $B$, one constructs
a uniformly computable function which takes (a code for) any such
object $B$ and returns (a code for) a corresponding object $A$.
This suggests, an alternative definition: A real $x$ is \emph{Martin-Löf random uniformly relative to}
a real $y$ if $x$ is not contained in any Martin-Löf null set \emph{uniformly computable from}
$y$ (see \cite{MiyabeRute13, Rute:2018} for formal definitions). 
While these two definitions agree for Martin-Löf randomness, they disagree for Schnorr randomness \cite{MiyabeRute13, Rute:2018}.  One needs to be cautious of this when
relativizing a result.  For example, the following is the correct (and most general) way to
relativize Theorem~\ref{thm:SR-monotone}. (This is actually the definition of relative Schnorr
randomness given in Rute \cite{Rute:2018}.)

\begin{thm}
\label{thm:SR-monotone-relative} For every oracle $a \in \mathbb{N}^\mathbb{N}$, the following are equivalent for a real $x\in[0,1]$.
\begin{enumerate}
\item The real $x$ is Schnorr random uniformly relative to $a$.
\item The supremum $\sup_{n}g^a_{n}(x)$ is finite for every increasing computable
sequence of continuous functions $g^a_{n}\colon[0,1]\rightarrow[0,\infty)$
uniformly computable in $a$ where $\int_{0}^{1}g^a_{n}(x)\,dx$ converges with a
rate of convergence uniformly computable in $a$.
\end{enumerate}
\end{thm}
Just as with Theorem~\ref{thm:SR-monotone}, we could apply the second 
Informal Principle to construct a proof of this theorem.
Using Bishop's proof, we can extract an algorithm which takes as input an oracle $a$,
a function $a \mapsto (g^a_n)_n$, and a function which maps $a$ to a rate
of convergence for $\int_{0}^{1}g^a_{n}(x)\,dx$.  The output of this algorithm is (a code for) a
null set $E^a$ for which $\{x:\lim_{n}g^a_n(x)=\infty\}\subseteq E^a$.

Moreover, Schnorr randomness behaves much better under uniform computability 
\cite{FranklinStephan10, Miyabe11, MiyabeRute13, Rute:2018},
solving many of the perceived flaws of Schnorr randomness. (For example,
Porter \cite[\S\S\S10.4.2]{Porter:2012} cataloged the four main objections to Schnorr
randomness. Each of these can be fixed by replacing ``computable'' with
``uniformly computable''.)

\subsection{Further investigations in randomness and analysis\label{subsec:Further-investigations}}

Besides the aforementioned topics, there are questions that only make sense
in the context of randomness. Fouché \cite{Fouche:2008} showed that
if $\phi\colon[0,1]\rightarrow\mathbb{R}$ is a Martin-Löf random Brownian
motion path (also called a \emph{complex oscillation}) then $\phi(1)$
is Martin-Löf random. Hoyrup and Rojas \cite[\S\S5.3]{Hoyrup.Rojas:2009b}
showed the converse also holds in the sense that if $x$ is Martin-Löf
random, then $\phi(1)=x$ for some Martin-Löf random complex oscillation $\phi$.
Rute \cite[Ex.~9.6]{Rute:2018} showed that this is true of Schnorr randomness as well. 
There are many more such results in randomness, e.g. \cite{Diamondstone.Kjos-Hanssen:2012,Porter.Culver:2015}. 
This is especially true in probability
theory, where one quickly passes between multiple representations of
the same object. A random sequence of independent fair coin tosses can be used to construct
a uniform random variable on $[0,1]$, a random walk on the integers,
a random graph, a random percolation model, and a number of other
random objects. It is important to know that randomness on one space
is (in some sense) equivalent to randomness on another. 

Four basic tools have been developed for this purpose. With the correct 
definitions, these theorems hold for both Schnorr and Martin-Löf randomness
(and often, but not always, hold for computable randomness).  Since the details
are a bit technical, we state them here vaguely with citations to the full 
theorems.\footnote{For the reader wishing to connect these results 
with Section~\ref{sec:Foundations}, we remark that when we say
``sufficiently effectively measurable'' it is sometimes sufficient for the map
$f\colon(\Omega,\mathbb{P})\rightarrow X$ to be Brouwer/Schnorr effectively 
measurable as in Definition~\ref{def:MLR-SR-eff-meas-fun}.  Other times
one must also require that the conditional probability map $x \mapsto \mathbb{P}(\cdot \mid f=x)$
be Brouwer/Schnorr effectively measurable as well.}
\begin{itemize}
\item (Randomness conservation) If $f\colon(\Omega,\mathbb{P})\rightarrow X$ is ``sufficiently effectively
measurable'' and $\omega\in\Omega$ is $\mathbb{P}$-random, then
$f(\omega)$ is random for the push-forward measure $\mathbb{P}_{f}$
(given by $\mathbb{P}_{f}(A)=\mathbb{P}(f^{-1}(A))$) \cite[Thms.~3.2, 4.1]{BienvenuPorter12},\cite[Prop.~5]{HoyrupRojas09}\cite{Rute16}\cite[Prop.~9.2]{Rute:2018}\cite[Thm.~2]{Bienvenu.Hoyrup.Shen:2017}.
\item (No randomness from nothing)  If $f\colon(\Omega,\mathbb{P})\rightarrow X$ is ``sufficiently effectively
measurable,'' and $x\in X$ is $\mathbb{P}_{f}$-random, then $x=f(\omega)$
for some $\mathbb{P}$-random $\omega\in\Omega$ \cite[Thm.~3.5]{BienvenuPorter12},\cite[Prop.~5]{HoyrupRojas09}\cite[Thm.~7]{Rute16}\cite[Cor.~9.5]{Rute:2018}\cite[Thm.~2]{Bienvenu.Hoyrup.Shen:2017}.
\item If $(\Omega,\mathbb{Q})$ is ``sufficiently effectively absolutely
continuous'' with respect to $(\Omega,\mathbb{P})$\footnote{
The measure $\mathbb{Q}$ is absolutely continuous
with respect to $\mathbb{P}$ if every $\mathbb{P}$-null set is $\mathbb{Q}$-null.} %
and $\omega\in\Omega$
is $\mathbb{Q}$-random, then $\omega$ is $\mathbb{P}$-random \cite[\S5.3]{HoyrupRojas09}\cite[footnote 3]{Rute16}.
\item (Van Lambalgen's theorem and its generalizations) Given $(\Omega_{1}\times\Omega_{2},\mathbb{P})$, the pair $(\omega_{1},\omega_{2})\in\Omega_{1}\times\Omega_{2}$
is $\mathbb{P}$-random if and only if $\omega_{1}$ is $\mathbb{P}_{1}$-random 
and $\omega_{2}$ is $\mathbb{P}(\cdot\mid\omega_{1})$-random
relative to $\omega_{1}$, assuming that $\mathbb{P}$ can be ``effectively
decomposed'' into the projection measure $\mathbb{P}_{1}$ on $\Omega_{1}$
and the family of conditional probabilities $\omega_{1}\mapsto\mathbb{P}(\cdot\mid\omega_{1})$
on $\Omega_2$ \cite{Takahashi:2008,Takahashi:2011a}\cite[Thm.~4]{Bauwens:2017}\cite[Thm.~5]{Bauwens.Shen.Takahashi:2017}\cite[Thm.~8.2]{Rute:2018}.  

A special case of this is where $\mathbb{P}$ is the product of two
independent measures $\mathbb{P}_1 \otimes \mathbb{P}_2$.  In this case, $(\omega_1,\omega_2)$ is $\mathbb{P}$-random if and only if $\omega_1$ is $\mathbb{P}_1$-random and $\omega_2$ is $\mathbb{P}_2$-random relative to $\omega_1$.  \cite[Thm.~6.9.1]{Downey.Hirschfeldt:2010}\cite[Thm.~3.4.6]{Nies:2009}\cite{MiyabeRute13,Miyabe11}.
\end{itemize}
For those interested in learning more about these new directions in
algorithmic randomness (at least with respect to Martin-Löf randomness),
we recommend Gács \cite{Gacs:aa}, Bienvenu, Gács, Hoyrup, Rojas,
and Shen \cite{Bienvenu.Gacs.Hoyrup.ea:2011}, Hoyrup and Rojas \cite{Hoyrup.Rojas:2009,Hoyrup.Rojas:2009a,Hoyrup.Rojas:2009b},
and Allen, Bienvenu, and Slaman \cite{Allen.Bienvenu.Slaman:2014}.
For computable randomness, see Rute \cite{Rute:,Rute16}.  For Schnorr
randomness, see \cite{Rute:2018, Rute:2013pd}.

\section{Randomness and the foundations of computable measure theory\label{sec:Foundations}}

We saw in Section~\ref{sec:Constructive-survey} that constructive
measure theory has been developed through a number of different constructive
and computable traditions \textemdash{} each tradition using slightly
different definitions, terminology, and techniques. This nonlinear
development, unfortunately, gives the outsider (and even the insider)
the appearance that ``a systematic general framework for computability 
in measure and integration theory still remains in its infancy'' \cite{Edalat:2009nx}.
This is far from the case.

In this section, we give a short presentation on the foundations of
computable measure theory. Our presentation shows that, while there
are many approaches to constructive/computable measure theory, they
are basically equivalent. One piece of evidence for this is that the
definitions of measurable set, measurable function, integrable function,
and almost uniform convergence in the computable and 
constructive literature basically agree.
Specifically, most definitions fall into three categories:
\begin{enumerate}
\item Point-free definitions.
\item Definitions which are well-defined outside of a Martin-Löf null set.
\item Definitions which are well-defined outside of a Schnorr null set.
\end{enumerate}
Moreover, all three categories are equi-computable, in the sense that, given
a computable object of one type, one can uniformly compute an equivalent
object of another type. As the descriptions of these categories suggest,
Martin-Löf and Schnorr randomness naturally arise out of these definitions
(although, in most cases there was no mention of randomness in the
original definitions).  (For simplicity, we only focus on whether our definitions are
\emph{computably} equivalent, ignoring whether they are \emph{constructively} equivalent.)

\subsection{Computable metric spaces and computable topology\label{subsec:Computable-metric-topology}}

To do computable analysis, one needs a good notion of a computable
space. The early constructivists restricted their work to Euclidean
space $\mathbb{R}^{d}$ or Cantor space $\{0,1\}^{\mathbb{N}}$. Later
work gradually incorporated compact and locally compact metric spaces,
separable Banach spaces, complete separable metric spaces, and finally
a wide variety of topological and abstract spaces. 

For this presentation, we will use \emph{complete separable metric spaces}
(also known as \emph{Polish spaces}). These spaces are sufficiently rich,
but still easy to work with. (Most random variables in probability
theory, for example, takes values in a complete separable metric space.)
\begin{defn}
A \emph{computable metric space} $X$ is a triple $(X,\rho,A)$ where 
$(X, \rho)$ is a complete seperable metric space, and $A\subseteq X$ is a dense 
indexed set $\{a_i\} \subseteq X$ (possibly with repetition) such that 
$i,j\mapsto \rho(a_i, a_j)$ is computable.  A point $x\in X$ is \emph{computable}
if there is a computable sequence $(i_n)$ such that for all $m < n$, 
$\rho(a_{i_n}, a_{i_m}) < 2^{-m}$ and $x = \lim_n a_{i_n}$. The sequence $(i_n)$
is called the \emph{Cauchy name} of $x$.

The \emph{effectively open sets} of $X$ are
computable sets of the form $U = \bigcup_i B(x_i, r_i)$ where $(x_i)_i$ is a 
computable sequence of points in $A$, $(r_i)_i$ is a computable sequence
of positive rationals, and $B(x_i, r_i) = \{x \in X : \rho(x,x_i) < r_i\}$.  
The \emph{effectively closed sets} are the complements of effectively open sets.

A partial map $f \colon D \subseteq X \to Y$ (where $Y$ is a computable metric space) is 
$\emph{computable}$ if there is a partial computable map 
$\Phi \colon \mathbb{N}^\mathbb{N} \to \mathbb{N}^\mathbb{N}$
which takes every $X$-Cauchy name for every $x \in D$ to a $Y$-Cauchy name of $f(x)$.
\end{defn}

\subsection{Computable measure spaces\label{subsec:Computable-measure-spaces}}
The set theoretic concept of a measure is so general that it is difficult to
distill it down to a computably representable form.  There are a few generally 
accepted approaches to do this.  One approach, which is simple and elegant, 
is to divorce the measure from the underlying topological structure 
of the space.  Any measure whose $\sigma$-algebra is countably generated can be 
represented with this approach.\footnote{There are two senses in which the 
sigma-algebra $\mathcal{A}$ of a measure $\mu$ is generated by a countable 
family of sets $\mathcal{R}$.
In a set theoretic sense, $\mathcal{A}$ is the minimum $\sigma$-algebra extending 
$\mathcal{R}$.  In a measure theoretic sense, $\mathcal{A}$ is the minimum 
$\mu$-complete sigma-algebra extending $\mathcal{R}$. That is $\mathcal{A}$ 
contains all $\mu$-null sets.  Since measure theory is normally ``up to a null set''
the differences are negligible.  However, for concreteness, when we say 
$\mathcal{R}$ generates $\mathcal{A}$, we mean the latter.  
When we speak later about Borel measures,
we will mean the completion of a Borel measure.} %
Recall, that a \emph{ring of sets} is a collection $\mathcal{R}$ of 
subsets of $X$ closed under union, intersection, empty set, and 
set difference.  If $X \in \mathcal{R}$, then $\mathcal{R}$ is a 
\emph{Boolean algebra}.  We say that a countable ring 
$\mathcal{R} = \{R_i\}$ is \emph{computable} if the index of 
$R_i \operatorname{\square} R_j$ is uniformly computable from $i$ and $j$ 
for $\square \in \{\cup, \cap, \smallsetminus\}$.

\begin{defn}\label{def:comp-meas-space} A \emph{computable $\sigma$-finite measure space} is a tuple
$(X,\mathcal{A},\mathcal{R}, \mu)$ where $\mathcal{R}=\{R_i\}$ is a computable 
ring of $X$ which
generates the $\sigma$-algebra $\mathcal{A}$ on $X$ and $i \mapsto \mu(R_i)$ is 
computable.  A \emph{computable finite measure space} is a computable 
$\sigma$-finite measure space $(X,\mathcal{A},\mathcal{R}, \mu)$ where $\mathcal{R}$ is a Boolean algebra of $X$.  
A \emph{computable probability space} is a \emph{computable finite measure space}
where $\mu(X) = 1$.
\end{defn}
This is the definition of Wu and Weihrauch \cite{Wu.Weihrauch:2006}.  
Also, there is no loss in loosening the Boolean operations on $\mathcal{R}$ 
up to $\mu$-a.e.\ equivalence. For example, if $R, S \in \mathcal{R}$, then 
we only require that there is a set $T \in \mathcal{R}$ such that $S \cup R = T$ $\mu$-a.e.
The fair-coin probability measure on $\{0,1\}^\mathbb{N}$ is computable with 
the Boolean algebra of cylinder sets.  The Lebesgue measure on $\mathbb{R}$ is
similarly computable with the ring of half-open rational intervals $(a,b]$.

Following Coquand and Palmgren \cite{Coquand.Palmgren:2002}, 
one can make this definition completely point-free by replacing the ring \emph{of sets} 
$\mathcal{R}$ with any countable \emph{algebraic} Boolean ring (without a unit) and the 
measure $\mu$ \emph{on sets} with a measure \emph{on the ring}.\footnote{A 
\emph{Boolean ring} is a commutative ring where $x^2 = x$.  
Ring multiplication and addition correspond to intersection and symmetric difference.
Union $x \cup y$ corresponds to $x + y + xy$.
A \emph{measure} $\mu$ on a ring $\mathcal{R}$ is a nonnegative function 
$\mu\colon \mathcal{R} \to [0,\infty)$ satisfying 
$\mu(x \cup y) = \mu(x) + \mu(y) - \mu(x \cap y)$ and $\mu(0) = 0$.} %
Similarly, a Boolean algebra \emph{of sets} is replaced with 
\emph{an algebraic} Boolean algebra, that is a Boolean ring with a unit $1$.
Another formal, point-free approach, based on the Danielle integral, 
was used by Bishop and Cheng \cite{Bishop.Cheng:1972,Bishop.Bridges:1985}.
Coquand and Palmgren \cite{Coquand.Palmgren:2002} and
Wu and Weihrauch \cite{Wu.Weihrauch:2006} showed that one can effectively
translate between the Danielle integral approach and the ring approach.

For simplicity, we will focus only on probability measure spaces $(X,\mathcal{A},\mathcal{R}, \mu)$, 
with an occasional footnote on 
finite and $\sigma$-finite measures.\footnote{The key observation of computable finite 
measures is that, with the exception of the zero measure, they are computable
probability measures scaled by a computable real.  The key observation of computable
$\sigma$-finite measures is that there is a computable partition $X_n$ of disjoint ring elements such 
that $X = \bigcup_{n = 0}^\infty X_n$ $\mu$-a.e., $\mu(X_n) > 0$, and the map $i \mapsto m$ such that 
$R_i \subseteq \bigcup_{n=0}^{m-1} X_n$ $\mu$-a.e. is computable \cite{Weihrauch.Tavana:2014}.  
Therefore a computable $\sigma$-finite measure space is just a disjoint union of uniformly computable 
finite measure spaces $(X_n, \mathcal{A}_n, \mathcal{R}_n, \mu_n)$.  Write $\mu = \sum_n \mu_n$.} %
Also, because these spaces do not have a topology, we cannot define
algorithmic randomness in the usual way.  We will show in Subsection~\ref{subsec:Forcing}
that one can still define Schnorr randomness for computable measure spaces 
via ``effectively generic ultrafilters.''

\subsection{The point-free approach to computable measure theory\label{subsec:Point-free-approach}}

Assume that $(X,\mathcal{A},\mathcal{R}, \mu)$ is a computable probability space 
and $Y$ is a computable metric space. Many of the objects of measure theory can
be described in a point-free way as points in a computable metric
space. As we discussed in Subsections~\ref{subsec:Russians} and \ref{subsec:Point-free}, 
this approach goes back to Šanin \cite{Sanin:1968dq} and has been developed by many others.
\begin{itemize}
\item The space $\mathrm{MSet}(X,\mu)$ of $\mathcal{A}$-measurable sets
(modulo a.e.~equivalence) is a computable metric space under the
metric $\rho(A,B)=\mu(A\triangle B)$. Call the computable points
in this space \emph{point-free effectively measurable sets}.\footnote{For 
a $\sigma$-finite measure space, $\rho(A,B)=\mu(A\triangle B)$ is a metric for the space of
finitely measurable sets.  The space of all measurable sets is given by the metric
$\rho(A,B) = \sum_n 2^{-n} \min\{1, \rho_{\mu_n}(A, B)\}$ where $\mu = \sum_n \mu_n$
as in the previous footnote.} %
\item The space $L^{0}(X,\mu)$ of measurable functions $f\colon (X,\mu)\rightarrow\mathbb{R}$
(modulo a.e.~equivalence) is a computable metric under the following
metric which describes convergence in measure.%
\footnote{This is just one of many computably equivalent metrics, also including the metric 
$\rho(f,g)  =\int\min\{|f-g|,1\}\,d\mu$ and the Ky-Fan metric.  
Again, for $\sigma$-finite measurable spaces, use the metric 
$\rho(f,g) =\sum_n 2^{-n} \min \{1, \rho_{\mu_n}(f,g)\}$.}
\begin{align*}
\rho(f,g) & =\int\frac{|f-g|}{1+|f-g|}\,d\mu
\end{align*}
Call the computable points in this space \emph{point-free effectively measurable functions}.
\item Similarly, the space $L^{0}(X,\mu;Y)$ of measurable functions $f\colon (X,\mu)\rightarrow Y$
(modulo a.e.~equivalence) is a computable metric under the following
metric (where $d_{Y}$ is the metric of $Y$).
\begin{align*}
\rho(f,g) & =\int\frac{d_{Y}(f,g)}{1+d_{Y}(f,g)}\,d\mu
\end{align*}
Call the computable points in this space \emph{point-free effectively measurable functions}
from $(X,\mu)$ to $Y$.
\item The space, $L^{p}(X,\mu)$ of $p$-integrable functions (modulo a.e.~equivalence)
for computable $1\leq p<\infty$ is a computable metric space under
the metric
\[
\rho(f,g)=\|f-g\|_{L^p}=\left(\int|f-g|^{p}\,d\mu\right)^{1/p}.
\]
We will call computable points in this space \emph{point-free effective $L^{p}$ functions}
or \emph{point-free effectively integrable functions} when $p=1$.\footnote{For 
$\sigma$-finite measures $\mu = \sum_n \mu_n$, there is also a space of 
locally $p$-integrable functions given by
the metric $\rho(f,g) = \sum_n 2^{-n} \min\{1,\|f-g\|_{L^p(\mu_n)}\}$.}
\end{itemize}
\begin{rem}
We have not yet mentioned which countable dense set to use for
each metric space. For measurable sets, use the Boolean algebra $\mathcal{R}$.  For
$L^0$ and $L^p$, use the set of rational step functions 
$\sum_{i=0}^{n-1} q_i \mathbf{1}_{R_i}$ where $R_0, \ldots, R_{n-1} \in \mathcal{R}$
is a partition of $X$ and $q_i \in \mathbb{Q}$.  For the $Y$-valued measurable functions,
use the same idea with the dense set $A$ of the computable metric space 
$Y$ taking the place of the rationals.

These above point-free definitions are equal to many others in the
literature. We list a few which are easily deducible from the definitions.
\end{rem}
\begin{itemize}
\item A set $A$ is point-free effectively measurable if and only if the
characteristic function $\mathbf{1}_{A}$ is point-free effectively
measurable (or point-free effectively $L^{p}$ or any computable $p$) 
\cite[Prop.~3.24, p.~41]{Rute:2013pd}.
\item A measurable function $f\colon(X,\mu)\rightarrow Y$ is point-free
effectively measurable if and only if for every effectively open set $U \subseteq X$,
there is a sequence of effectively measurable sets $A_0, A_1, \ldots$ 
(computable uniformly from the index of $U$) such that $f^{-1}(U) = \bigcup_i A_i$ $\mu$-a.e.
(This is basically the representation $\delta_\textrm{mfo}$ of \cite[Thm.~5.4]{Weihrauch:2017}.  
Also see Subsection~\ref{subsec:Computable-measures}.)
\item An $L^{p}(X,\mu)$ function $f$ (for computable $p\geq1$) is point-free
effectively $L^{p}$ if and only if $f$ is point-free effectively
measurable and $\|f\|_{L^p}$ is finite and computable
\cite[Prop.~3.20, p.~41]{Rute:2013pd}.
\item A bounded measurable function $f\colon (X,\mu)\rightarrow[0,1]$ 
is point-free effectively measurable if and only if $f$ is point-free effectively $L^{p}$
for any (and hence all) computable $p\geq1$
\cite[Prop.~3.20, p.~41]{Rute:2013pd}.
\end{itemize}
See Spitters \cite[Ch.~3]{Spitters:2002}\cite{Spitters:2006} for a 
modern constructive  treatment of this metric approach.
Moreover, Ko \cite[\S5.1]{Ko:1991} has given descriptions
of these classes via ``recursively approximable sets'' and ``recursively
approximable functions''. He also gave a characterization of the
effectively measurable functions via effective convergence in measure
\cite[Cor.~5.13]{Ko:1991} (see also Rute \cite[Prop.~3.15]{Rute:2013pd}). 
Edalat \cite{Edalat:2009nx}
gave a slightly different, but equivalent, characterization of bounded
measurable functions via interval-valued functions. 

Notice that in our point-free framework there is only one null set,
namely the equivalence class of the empty set. Even with such a limited
definition of ``null set,'' many almost everywhere
results can still be described in this framework. For example, two
sets $A$ and $B$ are a.e.\ equal if $\mu(A\triangle B)=0$. Also
a.e.\ convergence, while not a metrizable (or even topological) convergence,
can be defined within a point-free framework as follows by using effective
almost uniform convergence.

\begin{defn}
\label{def:pointfree-au-conv}A computable sequence of $Y$-valued, point-free $\mu$-effectively
measurable functions $(f_{k})_{k\in\mathbb{N}}$ converges \emph{point-free effectively almost uniformly}
to a function $f$ if there is a computable \emph{rate of almost uniform convergence}
$K\colon \mathbb{N} \times \mathbb{N} \to \mathbb{N}$ 
such that for all $n$,
\begin{equation}\label{eq:point-free-au}
\mu\{x\in X:\forall m\ \exists k > K(m, n)\ d_Y(f_{k}(x),f(x))>2^{-m}\}\leq 2^{-n}.
\end{equation}
\end{defn}

This definition was considered a constructive or effective version of 
\emph{almost everywhere (or almost sure) convergence} by Kosovski\u{\i} \cite{Kosovskii:1973a}, 
V'yugin \cite{Vyugin:1997}, Coquand and Palmgren \cite{Coquand.Palmgren:2002}
as well as many others.\footnote{Some authors use different but effectively 
equivalent definitions, e.g.\ replacing (\ref{eq:point-free-au}) with $\forall m\ \mu\{x\in X:\exists k > K(m, n)\ d_Y(f_{k}(x),f(x))>2^{-m}\}\leq 2^{-n}.$}
Recall, that if a sequence of functions converges almost uniformly, then it converges
almost everywhere. Egoroff's theorem says that the converse holds
for a probability space. However, Egoroff's theorem is not constructive.\footnote{Egoroff's theorem holds in Brouwer's measure theory because of the
fan principle \cite[\S\S 6.5.4]{Heyting:1956}. Bishop \cite[Ch.~7, Theorem 4]{Bishop:1967lq}
on the other hand, modified the definition of almost everywhere convergence,
making Egoroff's theorem trivial. Kosovski\u{\i} \cite[2.5.1]{Kosovskii:1970} 
gave a constructive counterexample to Egoroff's theorem, 
and Avigad, Dean, and Rute \cite{Avigad.Dean.Rute:2012}
show that Egoroff's theorem is equivalent to 2-$\mathsf{WWKL}$ over
$\mathsf{RCA}_{0}$.}  For that reason (and also the reason that Egoroff's theorem fails
for the convergence of continuously indexed families of functions 
$(f_{t})_{t\in[0,\infty)}$ as $t\to \infty$),
we choose to call this almost uniform convergence.\footnote{Ergoroff's 
theorem also fails, in general, for $\sigma$-finite measures. However,
one can easily develop a notion of ``local almost uniform convergence'' 
(and its effective analogue) which is classically equivalent to a.e.\ convergence.} 

Also note that the above definition of almost uniform convergence
is ``point-free'' in the sense that $\{x\in X:\forall m\ \exists k > K(m, n)\ d_Y(f_{k}(x),f(x))>2^{-m}\}$
is $\mu$-almost everywhere equal to $\{x\in X:\forall m\ \exists k > K(m, n)\ d_Y(g_{k}(x),g(x))>2^{-m}\}$
for any sequence $(g_{k})$ which is $\mu$-a.e.\ equal to $(f_{k})$ 
and any $g$ which is $\mu$-a.e.\ equal to $f$.

\subsection{Computable measures on computable metric spaces\label{subsec:Computable-measures}}

While the definition of a computable measure space in Definition~\ref{def:comp-meas-space} 
is both general and elegant, it requires imposing an arbitrary ring structure on the space,
effectively treating the space as zero-dimensional.
Now we will consider an alternative definition which
preserves the topological and metric structure of computable metric space $X$, 
while also inducing a computable metric structure on the space of probability measures on $X$.
For the majority of probability theory it is sufficient to work with Borel
probability measures on a Polish space.  For analysis, it is also common to work
with locally finite Borel measures on locally compact Polish spaces.  
Again, we will focus on the probability measure case, with an occasional footnote 
about locally finite measures.\footnote{Recall that a locally compact 
Polish space is the same as a locally compact second-countable Hausdorff space.  
A locally finite measure is one in which every point is contained in a neighborhood 
of finite measure.  For Borel measures on locally compact Polish spaces, locally-finite
measures are equivalent to $\sigma$-finite measures.  These are also called Radon measures.} %

If $X$ is a computable metric space, then the space $\mathcal{M}_{1}(X)$
of Borel probability measures on $X$ is a computable metric space
under the Levy-Prokhorov metric or the Wasserstein metric. (For the
Wasserstein metric, one must first modify $X$ to be a bounded metric
space.) The \emph{computable probability measures} $\mu\in\mathcal{M}_{1}(X)$
are the computable points in this metric space.\footnote{There
are also metrics one can use for the space $\mathcal{M}_\textnormal{loc}(X)$
of locally finite measures, e.g.\ Kallenberg~\cite[\S 15.7]{Kallenberg:1983}.} %
Equivalently, the computable probability measures can be described as follows.
\begin{enumerate}
	\item By an effective version of the Reisz representation theorem, $\mu\in\mathcal{M}_{1}(X)$
	is computable if and only if $f\mapsto\int f\,d\mu$ is a computable
	operator on bounded computable functions $f\colon X\rightarrow[0,1]$ \cite[Cors.~4.3.1, 4.3.2]{Hoyrup.Rojas:2009}.\footnote{For
	locally finite measures $\mu \in \mathcal{M}_\textnormal{loc}(X)$, $(f,K)\mapsto\int f\,d\mu$ 
	is a computable operator on pairs of computable functions $f\colon X\rightarrow[0,1]$ 
	and effectively compact sets $K$ such that $\operatorname{supp} f \subseteq K$.  
	See Bishop \cite[Ch.~6]{Bishop:1967lq}.}
	\item Using valuation theory, $\mbox{\ensuremath{\mu}}\in\mathcal{M}_{1}(X)$
	is computable if and only if $U\mapsto\mu(U)$ is a lower semicomputable
	operator on effectively open sets \cite[Thm.~4.2.1]{Hoyrup.Rojas:2009}.\footnote{For
    finite measures $\mu \in \mathcal{M}(X)$, one also needs $\mu(X)$ to be computable. 
    For locally finite measures $\mu \in \mathcal{M}_\textnormal{loc}(X)$, 
    see, for instance, Edalat~\cite{Edalat:2009nx}. } %
    (For Cantor space, this is equivalent to the map $\sigma \mapsto \mu([\sigma])$ being
    computable where $\sigma \in \{0,1\}^{<\mathbb{N}}$.)
\end{enumerate}
All of these approaches give the space $\mathcal{M}_{1}(X)$ the topology
of \emph{weak convergence}. For more on computable measures, see 
Schröder \cite{Schroder:2007kx} and Hoyrup and Rojas \cite{Hoyrup.Rojas:2009}.
For a constructive, point-free treatment of integral operators
and valuations, see Coquand and Spitters \cite{Coquand.Spitters:2009}.

Every computable probability space $(X, \mathcal{A}, \mathcal{R}, \mu)$ is
isomorphic to the computable measure $\nu$ on Cantor space $\{0,1\}^\mathbb{N}$
given by
\[\nu([\sigma]) = \mu\left(\bigcap_{\substack{i<|\sigma|\\\sigma(i)=1}} R_i\ \cap \ \bigcap_{\substack{i<|\sigma|\\ \sigma(i)=0}} R_i^c\right)\]
where $\mathcal{R} = \{R_i\}$.\footnote{
Similarly, every computable $\sigma$-finite measure space is isomorphic to a
measure on the locally compact space $\mathbb{N} \times \{0,1\}^\mathbb{N}$.} %

Conversely, given a computable probability measure $\mu$ on a computable metric space $X$, 
there is a computable sequence of radii $r_i > 0$, dense in $[0,1]$, such that $\mu\{d_X(x,a_i)=r_i\} =0$
for the dense set $A = \{a_j\}$ used to generate $X$.  In this way, the balls $B(a_j,r_i)$ 
form a basis of $X$ and $i,j \mapsto \mu(B(a_j,r_i))$ is computable.  The space
$(X,\mathcal{A},\mathcal{R},\mu)$ is a computable measure space where
$\mathcal{R}$ is the free Boolean algebra generated by these balls and $\mathcal{A}$
is the ($\mu$-completion of the) Borel sigma-algebra of $X$.\footnote{A 
similar construction can be done
for the locally finite measures on effectively locally compact computable metric spaces.  
This is basically the idea of Bishop's theory of profiles \cite[Ch.~6]{Bishop.Bridges:1985}.  
In this way we can
construct a ring $\mathcal{R}$ of open sets of computable measure which generates
the corresponding $\sigma$-finite measure space.} %
In this way we can extend
all the point-free definitions of the previous subsection to computable metric spaces with 
computable probability measures.

We can also now talk about the \emph{pushforward measure} $\mu_{f} \in \mathcal{M}_1(Y)$
of a measurable function $f:(X,\mu) \to Y$ given by $\mu_{f}(A)=\mu(f^{-1}(A))$.
This provides yet another characterization of point-free effectively measurable functions.
A function $f\colon(X,\mu)\rightarrow Y$ is point-free effectively measurable 
if and only if $\mu_f$ is computable
and the map $A\mapsto f^{-1}(A)$ is a computable map of type 
$\mathrm{MSet}(Y,\mu_{f})\rightarrow\mathrm{MSet}(X,\mu)$
\cite[Prop.~3.30, p.~43]{Rute:2013pd}.

\subsection{Two pointwise approaches\label{subsec:Pointwise-approaches}}

While the point-free approach is elegant, it is noticeably different from
classical measure theory, where a measurable function is actually
a function and a measurable set is actually a set. Also, there is
a certain conceptual advantage to thinking about functions as algorithms
which take a point in one space and assign it to a value in another
space.

There are two similar, but different pointwise variants of measure
theory in the constructive/computable literature. The first we will
call the \emph{Brouwer/Schnorr variant}, because it was the approach
used by Brouwer \cite{Brouwer:1919a,Heyting:1956} and it implicitly
uses Schnorr null sets. This variant is equivalent to approaches used
by Demuth \cite{Demuth.Kucera:1979}, Bishop \cite{Bishop:1967lq,Bishop.Cheng:1972,Bishop.Bridges:1985},
Pathak, Rojas, and Simpson \cite{Pathak.Rojas.Simpson:2014}, Rute
\cite{Rute:2013pd}, and Miyabe \cite{Miyabe:2013uq}. The second
variant we will call the \emph{Martin-Löf variant} since it was
used by Martin-Löf \cite{Martin-Lof:1970a} and implicitly uses Martin-Löf
null sets. This variant is equivalent to approaches given by Edalat
\cite{Edalat:2009nx}, Yu \cite{Yu:1994oz}, Brown, Giusto, and Simpson
\cite{Brown.Giusto.Simpson:2002}, Pathak \cite{Pathak:2009vn}, and
Hoyrup and Rojas \cite{Hoyrup.Rojas:2009a}. 

Assume $X$ and $Y$ are computable metric spaces and $\mu\in\mathcal{M}_{1}(X)$
is a computable measure.  In the previous subsection, we saw there is a countable 
Boolean algebra $\mathcal{R}$ of effectively open sets of computable measure which generates
this measure space.\footnote{Since
each set is open, by ``complement'' in $\mathcal{R}$ we mean the interior of
the complement.  This is acceptable, since, by construction, the boundary 
of each set in $\mathcal{R}$ is null.} %
The \emph{basic sets} are the elements of this Boolean algebra, where as the 
\emph{basic functions} $g\colon (X,\mu) \to Y$ are the step functions of the form
$g(x)=a_i$ if $x\in R_i$ where $R_0,\ldots,R_{n-1} \in \mathcal{R}$ is a 
finite partition of $\mathcal{R}$ and each $a_i$ is from the dense set generating $Y$.
(The basic functions are partial computable since we don't include the boundaries of the sets $R_i$.)
\begin{defn}
\label{def:MLR-SR-eff-meas-set}For a set $Q\subseteq X$,

\begin{itemize}
\item $Q$ is \emph{Martin-Löf effectively measurable} if there is a
computable sequence of basic sets $(R_{n})$ and a computable sequence
$(U_{n})$ of effectively open sets such that $\mu(U_{n})\leq2^{-n}$
and
\[
Q\triangle R_{n}\subseteq U_{n}.
\]
\item $Q$ is \emph{Brouwer/Schnorr effectively measurable} if, moreover,
$\mu(U_{n})$ is computable from $n$.
\end{itemize}
\end{defn}

Notice that a measure zero Martin-Löf effectively measurable set
is exactly a Martin-Löf null set, and a measure zero Brouwer/Schnorr
effectively measurable set is exactly a Schnorr null set.

In the following, let $f(x){\uparrow}$ denote that $x$ is not in the domain of $f$.
\begin{defn}
\label{def:MLR-SR-eff-meas-fun}For a partial function $f\colon X\rightarrow Y$,
where the metric of $Y$ is $d_Y$,

\begin{itemize}
\item $f$ is \emph{Martin-Löf effectively measurable} if there is a
computable sequence of basic functions $(g_{n})$ and a computable
sequence $(U_{n})$ of effectively open sets such that $\mu(U_{n})\leq2^{-n}$,
and
\[
\{x:f(x){\uparrow} \quad\lor\quad g_n(x){\uparrow} \quad\lor\quad d_Y(f(x),g_{n}(x))>2^{-n}\}\subseteq U_{n}.
\]
\item $f$ is \emph{Brouwer/Schnorr effectively measurable} if, moreover,
$\mu(U_{n})$ is computable from $n$.
\end{itemize}
\end{defn}

\begin{defn}
\label{def:MLR-SR-eff-L1}For a partial function $f\colon X\rightarrow\mathbb{R}$,

\begin{itemize}
\item $f$ is \emph{Martin-Löf effectively integrable} if there is a computable
sequence of basic functions $(g_{n})$ and a computable sequence $(U_{n})$
of effectively open sets such that $\mu(U_{n})\leq2^{-n}$, and
\[
\{x:f(x){\uparrow} \quad\lor\quad g_n(x){\uparrow} \quad\lor\quad |f(x) - g_{n}(x)|>2^{-n}\}\subseteq U_{n},
\]
and
\[
\int|f-g_{n}|\,d\mu\leq2^{-n}.
\]
\item $f$ is \emph{Brouwer/Schnorr effectively integrable} if, moreover,
$\mu(U_{n})$ is computable from $n$.
\end{itemize}
\end{defn}
The \emph{Martin-Löf} and \emph{Brouwer/Schnorr effective $L^p$ functions} are defined analogously.

\begin{defn}
\label{def:MLR-SR-eff-au-conv}Given a sequence of Martin-Löf effectively
measurable functions $f_{k}\colon X\rightarrow Y$ and a 
Martin-Löf effectively measurable function $f\colon X\rightarrow Y$,

\begin{itemize}
\item $f_{k}$ converges to $f$ \emph{Martin-Löf effectively almost uniformly}
if there is a computable \emph{rate of almost uniform convergence}
$K\colon \mathbb{N} \times \mathbb{N} \to \mathbb{N}$ 
and a computable sequence of effectively open sets $U_n$ where for all $n$, 
$\mu(U_n) \leq 2^{-n}$ and
\[
\{x\in X:\forall m\ \exists k\geq K(m, n)\ (f(x){\uparrow}\ \lor\  f_k(x){\uparrow}  \ \lor\   d_Y(f_{k}(x),f(x))>2^{-m})\}\subseteq U_n.
\]
\item $f_{n}$ converges to $f$ \emph{Brouwer/Schnorr effectively almost uniformly}
if, moreover, $\mu(U_{n})$ is computable from $n$.
\end{itemize}
\end{defn}

These pointwise versions allow us to treat measurable functions as true 
functions taking values $x$ and providing values $f(x)$.  For example, the Schnorr randomness
version of the Lebesgue differentiation theorem (Theorem~\ref{thm:SR-LDT}) can be strengthen to include 
a limit.

\begin{thm}[Pathak, Rojas, Simpson\cite{Pathak.Rojas.Simpson:2014}, Rute {{\cite[Thm.~4.10, p.~46, Thm.~12.3, p.~56]{Rute:2013pd}}}]
\label{thm:SR-LDT-2}The following are equivalent for a real $x\in[0,1]$.

\begin{enumerate}
\item The real $x$ is Schnorr random.
\item For every Brouwer/Schnorr effectively integrable function $f \colon [0,1] \to \mathbb{R}$,
\[\lim_{r\to0}\frac{1}{2r}\int_{x-r}^{x+r}f(y)\,dy = f(x).\]
\end{enumerate}
\end{thm}
\subsection{The equivalence of the three approaches and the connection with randomness\label{subsec:Equivalence-of-approaches}}

The three approaches \textemdash{} point-free, Martin-Löf, and Brouwer/Schnorr
\textemdash{} are all essentially equivalent. This next theorem is
stated for effectively measurable functions, but also holds for effectively measurable
sets, effective $L^{p}$ functions, and effective almost uniform convergence.
\begin{thm}\label{thm:equiv-of-3-approaches}
~

\begin{enumerate}
\item A Brouwer/Schnorr effectively measurable function is a Martin-Löf
effectively measurable function.
\item The equivalence class of a Martin-Löf effectively measurable function
is a point-free effectively measurable function.
\item If $f\colon (X,\mu)\rightarrow Y$ is a point-free effectively measurable
function given by a sequence $(f_{n})$ of basic functions such that
\begin{align}\label{eq:basic-func-converge}
\int\frac{d_{Y}(f,f_{n})}{1 + d_{Y}(f,f_{n})}\,d\mu\leq2^{-n},
\end{align}
then the partial function $\widetilde{f}\colon X\to Y$ 
given by $\widetilde{f}(x):=\lim_{n}f_{n}(x)$ (where the limit exists)
is a Brouwer/Schnorr effectively measurable function.  

Moreover the exceptional set $\{x:f_{n}(x) \text{ diverges}\}$ is a Schnorr null
set.  Also, if $f'_{n}$ is an alternate sequence of basic functions satisfying 
(\ref{eq:basic-func-converge}), then
$\{x:\lim_n f_{n}(x) \neq \lim_n f'_{n}(x)\}$ is a Schnorr null
set. (See Pathak, Rojas, Simpson \cite[Thm.~3.9]{Pathak.Rojas.Simpson:2014}, Rute \cite[Prop.~3.18]{Rute:2013pd}, Demuth and Ku\v{c}era \cite[Thm.~4.1]{Demuth.Kucera:1979}, and
Bishop and Bridges \cite[Props.~8.2, 8.3]{Bishop.Bridges:1985}.)
\end{enumerate}
\end{thm}
In particular, this above theorem implies that for every point-free
effectively measurable function $f$ and for every Schnorr random
$x$, there is a unique canonical value $\widetilde{f}(x):=\lim_{n}f_{n}(x)$.
This also shows that most theorems about the point-free and 
Martin-Löf effectively measurable functions
naturally generalize to the Brouwer/Schnorr effectively measurable functions.

Schnorr randomness is the weakest randomness notion for
this purpose.  Rute \cite[Thm.~12.19, p.~70]{Rute:2013pd} showed that 
there is no weaker randomness notion for which
Theorem~\ref{thm:equiv-of-3-approaches}(3) holds.  This again
demonstrates how Schnorr randomness naturally arises out of computable and 
constructive analysis 
(and that it is more than a coincidence that the Brouwer/Schnorr
definition is the pointwise definition adopted by most of 
the early constructivists).
\subsection{Other equivalent representations\label{subsec:Other-representations}}

Many of the constructive definitions in the literature are equivalent
to the Brouwer/Schnorr approach, including the definitions of Brouwer,
Demuth, and Bishop. However, there are a few caveats. First, Brouwer's and
Bishop's definitions are not computable, so one first needs to give
them a computable interpretation. Although Brouwer's definition of measurable
set is not defined on a measure one set,  it can be extended
to one \cite[\S\S6.2.2]{Heyting:1956}. This extension is equivalent to the Brouwer/Schnorr approach.
Demuth's definitions are restricted to the computable
reals, but these definitions naturally extend to the set of real numbers.\footnote{Also,
to be pedantic, in our definition of, say, Brouwer/Schnorr integrable function there
are $2^{2^{\aleph_0}}$ Brouwer/Schnorr effectively integrable functions.  
For example, any function
$f$ a.e.\ equal to $0$ is Brouwer/Schnorr effectively integrable
if $\{x : f(x) \neq 0\}$ is a Schnorr null set.
(In this way, our Brouwer/Schnorr representation is a multi-representation, whereby
each name corresponds to a set of objects.)
Whereas, some otherwise equivalent definitions of effectively integrable functions 
require that $f(x) = \lim_n g_n(x)$ for a 
computable sequence of simple functions and that the domain of $f$ is exactly the 
set of $x$ for which that limit converges.  In this case, there would only be countably many
Brouwer/Schnorr integrable functions, and every Brouwer/Schnorr integrable function
would be Borel-measurable.}

While verifying all these equivalences would take us too far afield, much
of the work can be done via the following lemma. Extending the definition
from Section~\ref{sec:Characterizing-randomness},
if $X$ and $Y$ are computable metric spaces and $\mu\in\mathcal{M}_1(X)$ is computable, 
then a partial function $f\colon(X,\mu)\rightarrow Y$
is \emph{almost everywhere computable} if there is a $\Pi_{2}^{0}$
subset $A\subseteq X$ of $\mu$-full measure such that $f\colon A\rightarrow Y$
is computable. (These are just the functions that are computable almost surely.
A definition in this regard, avoiding mention of $\Pi_{2}^{0}$ sets,
can be found in Rute \cite[Defs.~7.1, 7.4]{Rute:}.)
\begin{lem}[See Rute {{\cite[\S3, p.~36]{Rute:2013pd}}}]\label{lem:useful-conversions}
Let $X$ and $Y$ be a computable metric spaces and $\mu\in\mathcal{M}_1(X)$ be a computable probability measure.

\begin{enumerate}
\item Every computable function $f\colon X\rightarrow Y$ is Brouwer/Schnorr
effectively measurable.
\item Every almost everywhere computable function $f\colon X\rightarrow Y$
is Brouwer/Schnorr effectively measurable.
\item \label{item:useful-limits} If $(f_{n})$ is a computable sequence of 
Brouwer/Schnorr effectively measurable
functions such that (the equivalence classes of) $(f_{n})$ converge
point-free effectively almost uniformly (Definition~\ref{def:pointfree-au-conv}), then
$f_{n}$ converges Brouwer/Schnorr effectively almost uniformly and
the pointwise limit $f=\lim_{n}f_{n}$ is Brouwer/Schnorr effectively
measurable.
\item \label{item:ua-implies-SR}For computable $p$, the Brouwer/Schnorr effectively $L^{p}$ functions
$f\colon(X,\mu)\rightarrow\mathbb{R}$ are exactly the Brouwer/Schnorr
effectively measurable functions such that $\|f\|_{L^{p}}$ is computable.
\item The Brouwer/Schnorr effectively measurable sets $A$ are exactly the
sets such that $\mathbf{1}_{A}$ is Brouwer/Schnorr effectively measurable.
\end{enumerate}
\end{lem}
Also, inner and outer regularity along with Luzin's theorem provide convenient 
representations which are, respectively, equivalent to the Brouwer/Schnorr effectively measurable sets and the Brouwer/Schnorr effectively measurable functions.
\begin{itemize}
\item (Inner and outer regularity, \emph{Schnorr layerwise decidability})
The Brouwer/Schnorr effectively measurable sets $A$ are exactly those with
a computable sequence of effectively closed sets $C_{n}$ and effectively
open sets $U_{n}$ such that $C_{n}\subseteq A\subseteq U_{n}$, $\mu(C_{n})$
is computable in $n$, $\mu(U_{n})$ is computable in $n$, and $\mu(U_{n}-C_{n})\leq2^{-n}$.
(The sequence $C_{n}$ can also be modified to be compact \textemdash{}
in constructive analysis terminology these are called \emph{effectively located sets},
in computable analysis these are called \emph{computable sets}.) \cite[Prop.~3.22, p.~41]{Rute:2013pd}
\item (Luzin's theorem, \emph{Schnorr layerwise computability}) The Brouwer/Schnorr
effectively measurable functions $f\colon(X,\mu)\rightarrow Y$ are
exactly those with a computable sequence of closed (or even effectively
located/computable) sets $K_{n}$ such that $\mu(K_{n})\leq1-2^{-n}$
and $\mu(K_{n})$ is computable in $n$ and there is a sequence of
computable functions $f_{n}\colon K_{n}\rightarrow Y$ such that $f\upharpoonright K_{n}=f_{n}$.
\cite{Miyabe:2013uq}\cite[Prop.~3.21, p.~41]{Rute:2013pd}
\end{itemize}
For the constructive version of these results, see Spitters \cite{Spitters:2005}.
These above definitions can be modified so that they are equivalent to the Martin-Löf
effectively measurable sets and functions by removing the restriction that
$\mu(C_{n})$, $\mu(U_{n})$, and $\mu(K_{n})$ are computable. These
notions are called \emph{layerwise decidable sets} and 
\emph{layerwise comptuable functions} \cite{Hoyrup.Rojas:2009a,Hoyrup.Rojas:2009b}.
The idea is that if $K_{n}$ is the complement of the universal Martin-Löf
test, then to compute $f(x)$ for a Martin-Löf random $x$, one only
needs to know $x$ and an upper bound on the least $n$ such that
$x\in K_{n}$. This least $n$ is known as the \emph{randomness deficiency}
of $x$. Layerwise computability is a useful notion, because it gives
a very quick and intuitive method for showing that a function $f$ is Martin-Löf
effectively measurable.
\subsection{Other non-equivalent representations\label{subsec:Other-non-equivalent-representations}}

While most of the definitions in the literature align with the ones given above,
it should be mentioned that there are other useful representations.
For example, just as there are computable reals and reals computable from below,
there are natural representations of what it means for a measure, 
a real-valued measurable function, and a measurable set to be 
``computable from below'' (or ``from above'').
In particular, for measurable sets this captures the notion that measurable sets
form a locale (as mentioned in Subsection~\ref{subsec:Point-free}).  
See Weihrauch \cite{Weihrauch:2017}.
Also, often it is sufficient to represent a random variable, 
not as a measurable function, but only as a distribution (probability measure)
\cite{Muller:1999}.  Rute's work \cite{Rute16,Rute:2018} 
shows that it is convenient in randomness to
represent a measurable function $f \colon (X,\mu) \to Y$ by both a name for $f$
(as above) and also a name for the conditional probability map 
$y \mapsto \mu(\cdot | f = y)$ which is a measurable function of type 
$(Y,\mu_f) \to \mathcal{M}_1(X)$.  It appears that all natural examples of 
measurable functions are computable in this stronger sense.

Alternatively, in effective descriptive set theory \cite{Moschovakis:2009} one follows Borel's 
transfinite inductive definition to get effectively Borel measurable sets whose measures 
are hyperarithmetic reals.  
As Martin-Löf \cite{Martin-Lof:1970} first showed, this leads to its own notion of randomness.  
This ``higher randomness'' has since become its own area of study \cite[Ch.~14]{Chong.Yu:2015}.  
Coquand \cite{Coquand:2000} showed it is possible to reason constructively 
about measure theory in the Borel hierarchy using a hyperarithmetic definition of the reals.

\subsection{Computing effectively measurable functions}

So far our discussion of effectively measurable functions has been
a bit abstract. However, those interested in the foundations of computable
probability \textemdash{} including probabilistic algorithms and simulating
probabilistic processes \textemdash{} are right to ask the question,
``Can all this be implemented on a computer?'' The answer is yes!

Just as the effectively continuous functions $f\colon X\rightarrow Y$ are
the same as the computable functions from $X$ to $Y$ (which can
be implemented on a computer \textemdash{} in theory), effectively
measurable functions $f\colon(X,\mu)\rightarrow Y$ are the same as
recursively approximable functions (which can be implemented on a
computer). The definition goes back to Friedman and Ko \cite{Ko.Friedman:1982}.

Returning to the continuous case, assume $f\colon\{0,1\}^{\mathbb{N}}\rightarrow\mathbb{R}$
is computable. Then there is an algorithm $g\colon \mathbb{N}\times\{0,1\}^{\mathbb{N}}\rightarrow\mathbb{Q}$
which takes $x\in\{0,1\}^{\mathbb{N}}$ and $n\in\mathbb{N}$, and
returns an approximation $g(n,x)$ such that $|g(n,x)-f(x)|\leq2^{-n}$.
In short, this algorithm approximates $f$ in distance.

For a measurable function, we want an algorithm which approximates
$f$ both in distance and in probability.  (The following definitions naturally
generalize to any measurable function $f\colon (X,\mu)\to Y$.  See Bosserhoff \cite{Bosserhoff:2008}.)
\begin{defn}\label{def:rec-approx}
A measurable function $f\colon(\{0,1\}^{\mathbb{N}},\mu)\rightarrow\mathbb{R}$
is \emph{recursively approximable} if there is an algorithm $g\colon \mathbb{N}\times\{0,1\}^{\mathbb{N}}\rightarrow\mathbb{Q}$
which takes in $x\in\{0,1\}^{\mathbb{N}}$ and $n\in\mathbb{N}$,
and outputs an approximation $g(n,x)$ such that for all $n\in\mathbb{N}$,
\[
\mu\left\{ x:|g(n,x)-f(x)|>2^{-n}\right\} \leq2^{-n}.
\]
That is to say, for each $n$, there is a small probability $\leq2^{-n}$
that the algorithm will return a bad approximation. 
(To be clear, the algorithm need not know the approximation
is bad.)
\end{defn}
Notice that this definition is point-free in that it is invariant
under almost everywhere equivalence. Also, we could modify this definition
to allow $g$ to be partial. Assume $g$ is the same as above, except
that it is partial and 
\[
\mu\left\{ x: g(n,x){\uparrow}\ \lor\ |g(n,x)-f(x)|>2^{-n}\right\} \leq2^{-n}
\]
where $g(n,x){\uparrow}$ means $g$ does not halt with those inputs.
Then let $h(n,x)$ be the same as $g(n+1,x)$ except that after $g(n+1,x)$
has halted for at least $1-2^{n+1}$ $\mu$-measure of the $x$, we
set $h(n,x)=0$ for the rest.
\begin{thm}[Ko, Thm.~5.12 \cite{Ko:1991}]
The recursively approximable functions are the same as the point-free
effectively measurable functions.
\end{thm}

\begin{thm}
A measurable function $f\colon(\{0,1\}^{\mathbb{N}},\mu)\rightarrow\mathbb{R}$
is Brouwer/Schnorr effectively measurable if and only if $f$ is recursively
approximable with algorithm $g$ and 
\[
f(x)=\lim_{n\to\infty}g(n,x)\quad\text{on all \ensuremath{x} where \ensuremath{g(n,x)} converges}.
\]
\end{thm}
\begin{proof}
If $f$ is Brouwer/Schnorr effectively measurable and given by a sequence of basic functions $(g_n)$,
and a Schnorr test $(U_n)$, then set $g(n,x) = g_n(x)$ and we have
\[ \mu\left\{ x: |g(n,x)-f(x)|>2^{-n}\right\} \leq \mu(U_n) \leq 2^{-n}. \]
Hence it $f$ is recursively approximable.  Moreover, for all $x \notin \bigcap_n U_n$,
$\lim_n g(x, n) = f(x)$.  Finally, one can slightly modify $g(n,x)$ so that it does not
converge for any $x\in \bigcap_n U_n$.

Conversely, if $f$ is recursively approximable given as the limit of $g(n,x)$,
then the sequence $g(n,x)$ converges point-free effectively almost uniformly.  
By Lemma~\ref{lem:useful-conversions}(\ref{item:useful-limits}), $f$ is
Brouwer/Schnorr effectively measurable.
\end{proof}
Now, we give a concrete (and interesting) example of a recursively approximable function 
which is not just almost everywhere\ computable.
\begin{example}
Let $\lambda$ denote the fair-coin measure on $\{0,1\}^{\mathbb{N}}$. Let
$\bar{x}_\ell$ denote the frequency of $1$'s in the first $\ell$ bits of $x$, i.e.\ 
$\frac{1}{\ell}\sum_{k=0}^{\ell-1}x_{k}$.  Let  $f\colon\{0,1\}^{\mathbb{N}}\rightarrow\mathbb{R}$
be $f(x) = \sup_\ell \bar{x}_\ell$.
This function $f$ is recursively approximable as follows. Given $n$,
by standard probability estimates (e.g.\ martingale inequalities), there is some $m$ computable from $n$
such that 
\[ \lambda\left\{x: \sup_{\ell < m} \bar{x}_\ell = \sup_{\ell} \bar{x}_\ell \right\} > 1 - 2^{-n}.\]
Therefore, we can estimate $f(x)$ with $g(n, x) = \sup_{\ell < m(n)} \bar{x}_\ell$.
This estimate is correct (even exact) with probability at least $1-2^{-n}$.
Since $f$ is recursively approximate and $f(x) = \lim_n g(n,x)$, 
it is also Brouwer/Schnorr effectively measurable.

Notice, however, that $f$ cannot be almost everywhere computable. If it were,
then for almost all strings $x$ such that $f(x) < 2/3$, one could read finitely many bits
and be sure that $f(x)<2/3$.  However, it is impossible to know this almost surely 
from finitely bits of $x$.  
(Specifically, $\{x:f(x) < 2/3\}$ is a nowhere dense set of positive measure.)
\end{example}

\subsection{Obtaining Schnorr randomness through effective Solovay forcing\label{subsec:Forcing}}

In Subsection~\ref{subsec:Point-free} we saw that the Boolean algebra
of measurable sets modulo a.e.~equivalence seems to capture the intuitive
notion of randomness. In particular, by forcing with this poset (Solovay
random forcing) the resulting generics are exactly those reals which
are in every measure one set in the ground model. 

Now, we will consider the effective analogue of Solovay's forcing
construction. (Compare to the presentation in Jech~\cite[pp.~511--515]{Jech:2003}.)
Let $\mathcal{B}$ be the Boolean algebra of point-free effectively
measurable sets. A set $\mathcal{F}\subseteq\mathcal{B}\smallsetminus\{\varnothing\}$
is a \emph{filter} if it is upward-closed and closed under finite
meets (that is $A\cap B\in\mathcal{F}$ whenever $A\in\mathcal{F}$
and $B\in\mathcal{F}$. Moreover, $\mathcal{F}$ is an \emph{ultrafilter}
if for each $A\in\mathcal{B}$ either $A$ or its complement is in
$\mathcal{F}$. Say that an ultrafilter $\mathcal{G}$ is \emph{effectively generic}
if for every computable sequence $\mathcal{A}=(A_{n})$ of elements
in $\mathcal{G}$, if $\bigcap_{n}A_{n}$ is in $\mathcal{B}$, then
$\bigcap_{n}A_{n}$ is in $\mathcal{G}$. For a topological space, 
let $\mathcal{G}_\textnormal{cpt}$ 
(resp.\ $\mathcal{G}_\textnormal{cl}$) be the collection of all compact 
(resp.\ closed) sets whose equivalence class is in $\mathcal{G}$.

If we are working in a computable probability measure on a computable
metric space, then for every point-free effectively measurable set $A \in \mathcal{B}$ and
every Schnorr random $x$, by Theorem~\ref{thm:equiv-of-3-approaches}(3), 
there is a canonical value $\widetilde{\mathbf{1}_A}(x)$ which is either $0$ or $1$.
We write $x\in \widetilde{A}$ if $\widetilde{\mathbf{1}_A}(x)=1$ and $x\notin \widetilde{A}$
if $\widetilde{\mathbf{1}_A}(x)=0$.

\begin{prop}
Fix a computable probability measure $\mu$ on a computable metric space $X$
and let $\mathcal{B}$ be the Boolean
algebra of point-free effectively measurable sets.  The following are equivalent
for any collection $\mathcal{G} \subseteq \mathcal{B}$.

\begin{enumerate}
\item $\mathcal{G}$ is an effectively generic ultrafilter. 
\item $\mathcal{G}$ is an ultrafilter and $\bigcap\mathcal{G}_\textnormal{cl}=\bigcap\mathcal{G}_\textnormal{cpt} = \{x\}$ for some Schnorr random $x$.
\item $\mathcal{G}=\{A\in\mathcal{B} : x\in\widetilde{A}\}$ for some Schnorr random $x$. 
\end{enumerate}
\noindent The Schnorr randoms $x$ in (2) and (3) are the same.
\end{prop}

\begin{proof}
(1) $\rightarrow$ (2):  Since $\mathcal{G}$ is an effectively generic ultrafilter, $\mathcal{G}_\textnormal{cpt}$
is nonempty and contains subsets of arbitrarily small diameter. (Indeed,
one can effectively compute a countable cover of closed balls of computable
measure less than $\varepsilon$ \cite[Lemma~5.1.1]{Hoyrup.Rojas:2009}.
Then one can find compact subsets of these balls of measure arbitrarily
close to the measure of the ball. This gives a computable sequence
of compact sets $K_{n}$ of computable measure such that $\mu(\bigcup_{n}K_{n})=1$.
If $K_{n}\notin\mathcal{G}_\textnormal{cpt}$ for all $n$, then since $\mathcal{G}$
is an ultrafilter, the complement $U_{n}$ of $K_{n}$ is in $\mathcal{G}$
for all $n$. But $\bigcap_{n}U_{n}=\varnothing$ $\mu$-a.e.\ violating
the fact that $\mathcal{G}$ is effectively generic.) Since $\mathcal{G}$
is a filter, $\mathcal{G}_\textnormal{cpt}$ is closed under finite intersections,
and by compactness the intersection $G=\bigcap\mathcal{G}_\textnormal{cpt}$ is nonempty.
Since the sets of $\mathcal{G}_\textnormal{cpt}$ have arbitrarily small diameter,
the set $G$ is a singleton set $\{g\}$.  

We also have that $\mathcal{G}_\textnormal{cl}$ = $\{g\}$ 
since any effectively measurable, closed set $C \in \mathcal{B}$ which does not contain $g$
must be disjoint from some ball around $g$.  Therefore, $C$ is also 
disjoint from some effectively measurable, compact set 
$K \in \mathcal{G}_\textnormal{cpt}$ of arbitrarly small diameter containing $x$.
Hence, $\mathcal{G}_\textnormal{cl}$ is made up precisely of the effectively measurable, 
closed sets containing $g$.

To see that $g$ is Schnorr random, consider a Schnorr test $(U_{n})$
and let $(C_{n})$ be the complementary sequence of closed sets.
To show $g \notin \bigcap_n U_n$, it
it suffices to show $C_n \in \mathcal{G}_\textnormal{cl}$ for some $n$,
because then, $g \notin U_n$.
Assume for a contraction that $C_n \notin \mathcal{G}_\textnormal{cl}$ for all $n$.
Then $U_n \in \mathcal{G}$ for all $n$.  Since $(U_n)$ is a Schnorr test
and $\mathcal{G}$ is effectively generic,
$\varnothing=\bigcap_{n}U_{n}\in\mathcal{G}$
violating that $\mathcal{G}$ is a filter.
\medskip

(2) $\rightarrow$ (3): Assume $\mathcal{G}$ is an ultrafilter and 
$\bigcap\mathcal{G}_\textnormal{cl} = \{x\}$.
Since $\mathcal{G}$ is an ultrafilter, $\mathcal{G}_\textnormal{cl}$ is precisely the
collection of all effectively measurable, closed sets $C$ which contain $x$.
The rest follows from the following two regularity facts for an arbitrary effectively measurable set $A$
and a Schnorr random $x$ (which can be found in Rute \cite[Prop.~3.22, p.~41]{Rute:2013pd}).
\begin{itemize}
\item If $x \in \widetilde{A}$, then there exists a closed, effectively measurable set $C \subseteq A$ $\mu$-a.e.\ such that $x \in C$.
\item If $x \notin \widetilde{A}$, then there exists an open, effectively measurable set $U \supseteq A$
$\mu$-a.e.\ such that $x \notin U$.
\end{itemize}
\medskip

(3) $\rightarrow$ (1): Assume that $x$ is a Schnorr random and let $\mathcal{G}_{x}=\{A\in\mathcal{B} : x\in\widetilde{A}\}$.
To see that $\mathcal{G}_{x}$ is an effectively generic ultrafilter
it is enough to show the following. 

\begin{itemize}
\item $x\notin\widetilde{\varnothing}$.
\item For all $A,B\in\mathcal{B}$, if $x\in\widetilde{A}$ and $A\subseteq B$
$\mu$-a.e.\ then $x\in\widetilde{B}$.
\item For all $A\in\mathcal{B}$, if $x\in\widetilde{A}$ then $x\notin\widetilde{A^{c}}$.
\item For any computable sequence $(A_{n})$ from $\mathcal{B}$, if $x\in\widetilde{A_{n}}$
for all $n$ and $\bigcap_{n}A_{n}$ is effectively measurable, then
$x\in\widetilde{\bigcap_{n}A_{n}}$.
\end{itemize}
These results can all be found in Rute \cite[Prop.~3.28, p.~42]{Rute:2013pd}.
\end{proof}

This result shows that we can consistently extend Schnorr randomness to any
arbitrary computable probability space as in Definition~\ref{def:comp-meas-space},
even when there is not an underlying computable metric space.
\begin{defn}
For a computable probability space $(X,\mathcal{B},\mathcal{R},\mu)$, 
define a \emph{Schnorr random} to be an effectively generic ultrafilter $\mathcal{G}$
in the Boolean algebra of effectively measurable sets $\mathcal{B}$.
\end{defn}

Forcing is important in computability theory and proof theory. See
Shore \cite{Shore:} for a survey, Downey and Hirschfelt 
\cite{Downey.Hirschfeldt:2010} for examples of forcing in computably theory and randomness,
and Avigad \cite{Avigad:2004}
for examples of forcing in reverse and constructive mathematics. 
An alternative interpretation of ``effective Solovay forcing'' is due to Kautz \cite{Kautz:1991}\cite[\S\S7.2.5]{Downey.Hirschfeldt:2010}.
It is also known that forcing
with effectively closed sets of computable measure can be used to
construct Schnorr randoms with pathological properties; see,
for example, Yu \cite{Yu11}.

\bibliographystyle{alpha}
\bibliography{rand_analysis}

\end{document}